\documentclass[11pt]{article}
\usepackage{latexsym,amsmath,amsthm}
\usepackage{amssymb}
\usepackage{graphicx}
\usepackage{appendix}
\usepackage{tikz}
\usepackage[colorlinks=true, linkcolor=blue]{hyperref}
\usepackage{geometry}
\usepackage{url}
\usepackage{microtype}

\newtheorem{thm}{Theorem}[section]
\newtheorem{prop}[thm]{Proposition}
\newtheorem{cor}[thm]{Corollary}
\newtheorem{lemma}[thm]{Lemma}
\newtheorem{conj}[thm]{Conjecture}
\newtheorem{claim}[thm]{Claim}

\theoremstyle{definition}

\theoremstyle{definition}

\theoremstyle{definition}

\theoremstyle{definition}

\theoremstyle{definition}

\theoremstyle{definition}

\theoremstyle{remark}

\theoremstyle{remark}
\newtheorem{remark}[thm]{Remark}

\usepackage{dsfont}
\newcommand{\NN}{\ensuremath{\mathbb N}}

\newcommand{\RR}{\ensuremath{\mathbb R}}

\newcommand{\floor}[1]{\left \lfloor #1 \right \rfloor}
\newcommand{\ceil}[1]{\left \lceil #1 \right \rceil}

\newcommand{\E}{\mathbb{E}}
\renewcommand{\P}{\operatorname{Pr}}
\newcommand{\Poisson}{\operatorname{Po}}
\newcommand{\Bin}{\operatorname{Bin}}
\DeclareRobustCommand{\stirling}{\genfrac\{\}{0pt}{}}

\begin{document}
\title{Generalized Tuza's conjecture for random hypergraphs}

\author{Abdul Basit\thanks{Department of Mathematics,
Iowa State University, Ames IA; \texttt{abasit@iastate.edu}}
\and David Galvin\thanks{Department of Mathematics,
University of Notre Dame, Notre Dame IN; \texttt{dgalvin1@nd.edu}. 
Supported in part by the Simons Foundation.
}}

\maketitle

\begin{abstract}
A celebrated conjecture of Tuza states that in any finite graph the minimum size of a cover of triangles by edges is at most twice the maximum size of a set of edge-disjoint triangles. For an $r$-uniform hypergraph ($r$-graph) $G$, let $\tau(G)$ be the minimum size of a cover of edges by $(r-1)$-sets of vertices, and let $\nu(G)$ be the maximum size of a set of edges pairwise intersecting in fewer than $r-1$ vertices. Aharoni and Zerbib proposed the following generalization of Tuza's conjecture: $$ \text{For any $r$-graph $G$,  $\tau(G)/\nu(G) \leq \ceil{(r+1)/2}$.} $$

Let $H_r(n,p)$ be the uniformly random $r$-graph on $n$ vertices. We show that, for $r \in \{3, 4, 5\}$ and any $p = p(n)$, $H_r(n,p)$ satisfies the Aharoni-Zerbib conjecture with high probability (i.e., with probability approaching 1 as $n \rightarrow \infty$). We also show that there is a $C < 1$ such that, for any $r \geq 6$ and any $p = p(n)$, $\tau(H_r(n, p))/\nu(H_r(n, p)) \leq C r$ with high probability. Furthermore, we may take $C < 1/2 + \varepsilon$, for any $\varepsilon > 0$, by restricting to sufficiently large $r$ (depending on $\varepsilon$).
\end{abstract}

\section{Introduction}

For a finite graph $G$, $\nu_t(G)$, the {\em triangle matching number}, is the maximum size of a set of edge-disjoint triangles in $G$, and $\tau_t(G)$, the {\em triangle cover number}, is the minimum size of a set of edges $C$ such that every triangle in $G$ contains a member of $C$. It is immediate from the definitions that $\nu_t(G) \leq \tau_t(G) \leq 3\nu_t(G)$. A celebrated conjecture of Tuza states:
\begin{conj}[Tuza~\cite{tuza90}]
\label{conj:tuza}
For any graph $G$, $\tau_t(G) \le 2\nu_t(G)$.
\end{conj}
If true, Conjecture~\ref{conj:tuza} would be sharp, as is seen by taking $G$ to be $K_4$ or $K_5$ (or a disjoint union of these), and is close to sharp in many other cases (see~\cite{bk16, hkt12}). The best known general bound is $\tau_t(G) \leq \frac{66}{23}\nu_t(G)$, due to Haxell~\cite{haxell99}. 
In a recent  development, Kahn and Park~\cite{kp20}, and (independently) Bennett, Cushman and Dudek~\cite{bcd20} show that Tuza's conjecture is true for the Erd\H{o}s-R\'enyi random graph $G(n, p)$.
\begin{thm}[\cite{kp20, bcd20}]
For any $p = p(n)$, $\tau_t(G(n, p)) \leq 2\nu_t(G(n, p))$ w.h.p.\footnote{where ``with high probability'' means with probability tending to 1 as $n \rightarrow \infty$.}
\end{thm}

In this paper, we are interested in a generalization of Conjecture~\ref{conj:tuza} to $r$-uniform hypergraphs, which was proposed by Aharoni and Zerbib~\cite{az20} (see Conjecture~\ref{conj:generaltuzar}). We show that, for small $r$, the \emph{uniformly random $r$-graph} satisfies the Aharoni-Zerbib conjecture w.h.p., and we obtain non-trivial bounds for larger $r$.

\subsection{Notation and definitions}

A hypergraph $G$ is a collection of distinct subsets (referred to as {\em edges}) of a set $V(G)$ (the {\em vertex set}). Hence, $|G|$ is the number of edges of the hypergraph $G$.
If all edges have cardinality $r$, then $G$ is {\em $r$-uniform}. Throughout, we restrict our attention to $r$-uniform  hypergraphs (or {\em $r$-graphs} for short). For a set $\sigma \subseteq V(G)$ with $|\sigma| = r - 1$, the {\em neighborhood} of $\sigma$, $N_G(\sigma)$, is the set of $v \in V(G)$ such that $\sigma \cup \{v\} \in E(G)$. The {\em co-degree} of $\sigma$, denoted by $d_G(\sigma)$, is the size of its neighborhood. When $G$ is clear from context, we omit the subscript and write $N(\sigma)$ and $d(\sigma)$. We define $D(G)$ to be the collection of $(r-1)$-sets that have non-zero co-degree, i.e., $D(G) = \{\sigma \in \binom{V(G)}{r-1}: d(\sigma) > 0\}$. Equivalently, $D(G)$ is the union over edges $e$ of the $(r-1)$-subsets of $e$.

A {\em matching} in an $r$-graph $G$ is a set of pairwise disjoint edges, and the {\em matching number} $\nu(G)$ is the maximum size of a matching in $G$. A {\em cover} of $G$ is a set of vertices intersecting all edges of $G$, and the {\em covering number} $\tau(G)$ is the minimum size of a cover. Clearly we have
\begin{equation*}
    \nu(G) \leq \tau(G) \leq r\nu(G).
\end{equation*}
More generally, an {\em $m$-matching $M$} in $G$ is a set of edges such that any two edges in $M$ intersect in fewer than $m$ vertices, and an {\em $m$-cover} $C$ of $G$ is a collection of $m$-sets of vertices such that every edge in $G$ contains at least one member of $C$. So a 1-matching is a matching, and a 1-cover is a cover. 
The {\em $m$-matching number}, denoted by $\nu^{(m)}(H)$, is the maximum size of an $m$-matching in $H$, and the {\em $m$-covering number}, denoted by $\tau^{(m)}(H)$, is the minimum size of an $m$-cover of $H$.

For a set $X$ and a natural number $k$ we use $\binom{X}{k}$ for the set of $k$-sets of $X$.
Let $G^{(m)}$ be the  hypergraph whose vertex set is $\binom{V(G)}{m}$ and whose edge set is $\left\{ \binom{e}{m}: e \in G\right\}$. Note that $M \subseteq G$
is an $m$-matching in $G$ if and only if $\left\{ \binom{f}{m}: f \in M\right\}$ is a matching in $G^{(m)}$, and a collection of $m$-sets $C$ is an $m$-cover of $G$ if and only if $C$ is a cover of $G^{(m)}$. That is, $\nu^{(m)}(G) = \nu(G^{(m)})$ and $\tau^{(m)}(G) = \tau(G^{(m)})$, implying
\begin{equation*}
    \nu^{(m)}(G) \leq\tau^{(m)}(G) \leq \binom{r}{m}\nu^{(m)}(G).
\end{equation*}

Given a (2-)graph $G$, let $T(G)$ be the $3$-graph with vertex set $V(G)$ whose edges are triples of vertices forming triangles in $G$. Conjecture~\ref{conj:tuza} can be rephrased as: for any graph $G$, $\tau^{(2)}(T(G)) \leq 2\nu^{(2)}(T(G))$. Aharoni and Zerbib conjectured that the same inequality holds for any 3-graph $G$.
\begin{conj}[Aharoni-Zerbib~{\cite[Conjecture~1.2]{az20}}]
\label{conj:generaltuza3}
For any 3-graph $G$, $$\tau^{(2)}(G) \le 2\nu^{(2)}(G).$$
\end{conj}
They also conjectured that a similar phenomenon holds much more generally. 
\begin{conj}[Aharoni-Zerbib~{\cite[Conjecture~1.10]{az20}}]
\label{conj:generaltuzar}
For any $r$-graph $G$, \[ \tau^{(r-1)}(G) \leq \ceil{\frac{r+1}{2}}\nu^{(r-1)}(G).\]
\end{conj}

If true, Conjecture~\ref{conj:generaltuzar} would be tight; for every $r$, there exists $r$-graphs $G$ with $\nu^{(r-1)}(G) = 1$ and $\tau^{(r-1)}(G) = \ceil{\frac{r+1}{2}}$. The complete $r$-graph on $r+1$ vertices is one such example; see {~\cite[Proposition~3.7]{az20}} for details. We refer to~\cite{bmssz22, gs20} for partial results on Conjecture~\ref{conj:generaltuzar} and its fractional variants.

\subsection{Results and organization} \label{subsec-results-org}

We study Conjecture~\ref{conj:generaltuzar} for random hypergraphs. Specifically, let $H_r(n, p)$ be the $r$-graph on a set $V$ of $n$ labelled vertices, where each $e \in \binom{V}{r}$ is an edge with probability $p=p(n)$, independent of other $e' \in \binom{V}{r}$ (so $H_2(n, p)$ is the Erd\H{o}s-R\'enyi, or Gilbert, random graph $G(n,p)$). For $3 \leq r \leq 5$ and with $G = H_r(n,p)$, we verify Conjecture~\ref{conj:generaltuzar}.
\begin{thm}
\label{thm:mainsmall}
For any $3 \leq r \leq 5$ and any $p = p(n)$, \[ \tau^{(r-1)}\left(H_r(n, p)\right) \leq \ceil{\frac{r+1}{2}} \nu^{(r-1)}\left(H_r(n, p)\right)\quad \mbox{ w.h.p. }\]
\end{thm}
\begin{remark}
In fact, our proof of Theorem~\ref{thm:mainsmall} shows, numerically, that the truth is strictly less than $\ceil{\frac{r+1}{2}}$. See Appendix~\ref{app:smallr} for details. This mirrors the situation for the original Tuza conjecture, where there is $C<2$ such that for all $p$, $\tau(G_{n,p}) \leq C\nu(G_{n,p})$ with high probability (see \cite{kp20}).
\end{remark}

We also show that our methods can be extended for larger $r$, to give a non-trivial upper bound on $\tau^{(r-1)}\left(H_r(n, p)\right)/\nu^{(r-1)}\left(H_r(n, p)\right)$ (w.h.p.). Unfortunately, while we believe Conjecture~\ref{conj:generaltuzar} to be true, our methods do not seem sufficient to obtain it.
\begin{thm}
\label{thm:mainlarge}
There exists an absolute constant $C < 1$ such that, for any $r \geq 6$ and any $p = p(n)$, 
\begin{equation} \label{eq-mainlarge1}
\tau^{(r-1)}\left(H_r(n, p)\right) \leq C r  \nu^{(r-1)}\left(H_r(n, p)\right)\quad \mbox{ w.h.p. }
\end{equation}
Furthermore, for every $\varepsilon>0$, there is $r_0 = r_0(\varepsilon)$ such that, for all $r \geq r_0$ and $p = p(n)$, 
\begin{equation} \label{eq-mainlarge2}
\tau^{(r-1)}\left(H_r(n, p)\right)  \leq \left(\frac{1}{2}+\varepsilon\right)r \nu^{(r-1)}\left(H_r(n, p)\right) \mbox{ w.h.p.}
\end{equation}
\end{thm}

\begin{remark}
For $r = 6$ we show $C$ may be taken to be 0.781; for $7 \leq r \leq 270$, $C$ may be taken to be 0.938; for $r \geq 271$, $C$ may be taken to be 0.747. With more involved analysis these numbers could certainly be improved; for example, for $r \geq 1000$, $C$ can be taken to be 0.6964, and by appeal to a numerical optimization tool such as {\tt Mathematica} we can obtain an improved bound for any particular $r$. These calculations suggest, for example, that for $r \geq 84$, we may take $C<0.60$, and verify that for $r\geq 12$ we may take $C<0.70$. Since it is very unlikely that our present methods are optimal, we forgo this more detailed analysis, and just make a short report on the process in Appendix~\ref{app:computations}. 
\end{remark}

\medskip

Let $d = d(r,n,p) = (n - (r-1))p$ be the expected co-degree of an $(r-1)$-set of vertices of $H_r(n, p)$.
We consider three ranges of $d$, each requiring different approaches. Our methods are inspired by those of Kahn-Park~\cite{kp20}, and some of the following results are immediate generalizations of the results therein (Theorems~\ref{thm:small-p},~\ref{thm:matching-large-p},~and~\ref{thm:matching-medium-p}). To keep the paper self-contained, we give complete proofs.
Our main contribution consists of obtaining bounds, for $r > 3$, on both $\tau^{(r-1)}(H_r(n,p))$ and $\tau^{(r-1)}(G)$ for arbitrary $r$-graph $G$  (Theorems~\ref{thm:cover-absolute},~\ref{thm:cover-medium-p-small-r},~and~\ref{thm:cover-medium-p-large-r}).

That $\tau^{(r-1)}(H_r(n,p)) \geq \nu^{(r-1)}(H_r(n,p))$ is trivial. For small $d$, we obtain the following optimal bound.
\begin{thm}
\label{thm:small-p} 
For every $r \geq 3$, if $d \leq 1/(r-1)$, then w.h.p.  
$$
\nu^{(r-1)}(H_r(n,p)) \sim \tau^{(r-1)}(H_r(n,p)).
$$
\end{thm}

Note that throughout the paper, we use the standard Bachmann–Landau notation to describe the limiting behavior of functions as $n \rightarrow \infty$. We also write $A \sim B$ if $A/B \rightarrow 1$ as $n \rightarrow \infty$, and $A \lesssim B$ if $A \leq (1 - o(1))B$ as $n \rightarrow \infty$.

For large $d$, we use a variant of Pippenger's Theorem~\cite{kahn96pippenger} to show that there is an $(r - 1)$ matching that contains almost every $(r - 1)$-set, implying the following. 
\begin{thm}
\label{thm:matching-large-p}
For every $r \geq 3$, if $d = \omega(1)$, then w.h.p. $\nu^{(r-1)}(H_r(n,p)) \sim \frac{1}{r} \binom{n}{r-1}$.
\end{thm}

To bound $\tau^{(r-1)}$, we show that known constructions for the hypergraph Tur\'an problem~\cite{keevash2011} can be adapted to obtain $(r-1)$-covers for arbitrary $r$-graphs. Recall from earlier that $D(G)$ is the union over edges $e$ of $G$ of the $(r-1)$-subsets of $e$.
\begin{thm}
\label{thm:cover-absolute}
Let $r \geq 3$ and $l \leq r/2$. Then, for any $r$-graph $G$, \[ \tau^{(r-1)}(G) \leq \frac{|D(G)|}{2}\left[ \frac{1}{l} + \left(3 + \frac{r-1}{l-1}\right)\left(1 - \frac{1}{l}\right)^{r-1}\right]. \]
Additionally, 
\[ \tau^{(r-1)}(G) \leq
\begin{cases} 
\frac{1}{2}|D(G)| &\text{if }r = 3, \\
\frac{4}{9}|D(G)| &\text{if }r = 4,\\
\frac{5}{16}|D(G)| &\text{if }r = 5. \\
\end{cases}\]
\end{thm}

Theorems~\ref{thm:matching-large-p} and \ref{thm:cover-absolute} together almost immediately imply Theorems~\ref{thm:mainsmall} and \ref{thm:mainlarge} for $d = \omega(1)$. In fact, for large~$r$, we obtain the significantly stronger $\tau^{(r-1)}(H_r(n,p))/\nu^{(r-1)}(H_r(n,p)) \leq (1 + o(1)) \ln r$ w.h.p. The details are given in Appendices~\ref{app:smallr} (Theorem ~\ref{thm:mainsmall} for $d = \omega(1)$), and \ref{app:largerlarged} (Theorem \ref{thm:mainlarge} for $d = \omega(1)$).

Finally, for $d = \Theta(1)$, extending ideas of Kahn-Park~\cite{kp20}, we obtain the following.
\begin{thm}
\label{thm:matching-medium-p}
For every $r \geq 3$, if $d = \Theta(1)$, then w.h.p.
\[ \nu^{(r-1)}(H_r(n,p)) > (1 + o(1)) \frac{1}{r} \alpha_r(d) \binom{n}{r-1}\]
where
\[
\alpha_r(d) = 1 - \left(\frac{1}{(r-1)d + 1}\right)^{1/(r-1)}.
\]
\end{thm}
\begin{thm}
\label{thm:cover-medium-p-small-r}
Let $r \in \{3, 4, 5\}$. If $d = \Theta(1)$, then w.h.p. 
\[
\tau^{(r-1)}(H_r(n,p)) \leq (1 + o(1)) \beta_r(d) \binom{n}{r-1}
\]
where
\[ \beta_r(d) = \begin{cases} 
\frac{1}{2} \left[1 - \exp\left(-\frac{d}{2}\left(1 + e^{-d}\right)\right) \right]&\text{if }r = 3, \\
\frac{4}{9} \left[1 - \exp\left(-\frac{d}{3}\left(2 + e^{-2d}\right)\right)\right] &\text{if }r = 4,\\
\frac{5}{16} \left[1 - \exp\left(-\frac{d}{2}\left(1 + e^{-d}\right)\right) \right]  &\text{if }r = 5. \\
\end{cases} \]
\end{thm}

Theorem~\ref{thm:mainsmall} now follows from Theorems~\ref{thm:matching-medium-p} and \ref{thm:cover-medium-p-small-r} combined with the fact that $\beta_r(d)/(\alpha_r(d)/r) \leq \ceil{(r+1)/2}$ (see Appendix~\ref{app:smallr}). 
For larger $r$, we use Theorem~\ref{thm:matching-medium-p} and the simpler bound from Theorem~\ref{thm:cover-absolute}, which is sufficient to obtain Theorem~\ref{thm:mainlarge} (see Appendix~\ref{app:largermediumd}). This comes at the cost of a worse constant. However, we show that the methods used to obtain Theorem~\ref{thm:cover-medium-p-small-r} can be extended to larger $r$.

For a uniform random partition of an $(r-1)$-set into $l$ blocks (i.e., each element is placed into one of the $l$ blocks, independent of other elements), let $\zeta_1 := \zeta_1(r, l)$ be the probability that there are least two elements in each block, and let $\zeta_2 = \zeta_2(r, l) = 1 - \zeta_1(r, l)$. We obtain the following.
\begin{thm}
\label{thm:cover-medium-p-large-r}
Let $ r \geq 3$ and $l \leq r/2$. If $d = \Theta(1)$, then w.h.p.
$$\tau^{(r-1)}(H_r(n,p)) < (1 + o(1)) \binom{n}{r-1} \psi_{r, l}(d),$$
where
$$\psi_{r, l}(d) = \frac{\zeta_1}{2l} \left(1 - \exp\left(-\frac{d}{2} \left(1 + e^{-d}\right)\right) \right) + \left(\frac{\zeta_2}{l} + \left(1 - \frac{1}{l}\right)^{r-1}  \right) \left(1 - e^{-d}\right).$$
\end{thm}
In Appendix~\ref{app:r6}, we use Theorem~\ref{thm:cover-medium-p-large-r} to obtain a better bound for $r = 6$. Finally, in Appendix~\ref{app:computations}, we give computational evidence that Theorem~\ref{thm:cover-medium-p-large-r} implies better bounds for all $r \geq 7$.

\paragraph{Conventions} 
For the rest of this paper, unless otherwise noted, we will treat $r$ as a fixed constant. We write $\nu(G)$ and $\tau(G)$ for $\nu^{(r-1)}(G)$ and $\tau^{(r-1)}(G)$; when the $r$-graph $G$ is clear from context we simply write $\tau$ and $\nu$. We fix $H := H_r(n, p)$ to be the random $r$-graph and use $G$ for a general $r$-graph. Let $d = (n - (r-1))p$ be the expected co-degree of an $(r-1)$-set of vertices of $H$.

\paragraph{Outline}
In Section~\ref{sec:probabilistic}, we collect some standard probabilistic tools that we use throughout. Section~\ref{sec:coupling} introduces \emph{breadth-first trees} and gives a coupling which is central to the proofs of Theorems~\ref{thm:small-p}, and \ref{thm:matching-medium-p}--\ref{thm:cover-medium-p-large-r}. Sections~\ref{sec:small-p}, \ref{sec:matching-large-p}, and \ref{sec:matching-medium-p} give the proofs of Theorems~\ref{thm:small-p}, \ref{thm:matching-large-p}, and \ref{thm:matching-medium-p}, respectively. Section~\ref{sec:covers} is devoted to the proofs of Theorems~\ref{thm:cover-absolute}, \ref{thm:cover-medium-p-small-r}, and \ref{thm:cover-medium-p-large-r}. As noted above, that Theorems~\ref{thm:small-p}--\ref{thm:cover-medium-p-large-r} imply Theorem~\ref{thm:mainsmall} and Theorem~\ref{thm:mainlarge} is shown in Appendices~\ref{app:smallr}--\ref{app:r6}. In Appendix~\ref{app:computations}, based on computational evidence, we present some better bounds that could be obtained with a more involved analysis.

\paragraph{Acknowledgments}
We would like to thank Mohsen Aliabadi, Daniel McGinnis, and Shira Zerbib for extensive discussions on an earlier version of this paper.

\section{Probabilistic Tools}
\label{sec:probabilistic}

We need the following standard concentration facts. The first is the Chernoff-Hoeffding inequality (see e.g~{\cite[Theorem 21.6]{friezekaronski}}).
\begin{thm}
\label{thm:chernoff}
If $X$ is binomial with $\E X = \mu$, then for all $t \geq 0$
\[ \P(X \geq \mu + t) \leq \exp[-t^2/(2(\mu + t/3))], \]
and for $t \leq \mu$,
\[ \P(X \leq \mu - t) \leq  \exp[-t^2/(2(\mu - t/3))]. \]
\end{thm}

The second fact is McDiarmid’s inequality, or the Hoeffding-Azuma inequality (see e.g.~{\cite[Lemma 1.2]{mcdiarmid1989}}).
\begin{thm}
\label{thm:azuma}
Let $X_1, \dots, X_l$ be independent random variables, with $X_k \in A_k$ for each $k$. Suppose the (measurable) function $f:\prod A_k \rightarrow \RR$ satisfies, for each $k$,
\begin{equation}
\label{eq:azuma1}|f(X) - f(X')| \leq c_k \end{equation}
whenever $X = (X_i: i \in [l])$ and $X' = (X'_i: i \in [l])$ differ only in their $k^{th}$ coordinates.

Then, for any $t > 0$, \[ \P(|f - \E f| \geq t) \leq 2\exp[-2t^2/\sum c^2_k]. \]
\end{thm}

We use Theorem~\ref{thm:azuma} with $l = \binom{n}{r}$ and $X_i = \mathbf{1}_{e_i \in H}$, where $\{e_i: i \in [l]\} = \binom{V(H)}{r}$ (i.e., $X$ is the random hypergraph $H$).
Then we have \begin{equation}
    \mbox{if $f$ is Lipschitz (i.e., satisfies \eqref{eq:azuma1} with $c_k = 1$)  and $\E f = \omega( n^{r/2})$, then $f \sim \E f$ w.h.p. }
\end{equation}

We denote by $\delta(X, Y)$ the {\em total variation distance} between discrete random variables $X$ and $Y$. That is, $\delta(X, Y)$ is half the $\ell_1$ distance between the distributions of $X$ and $Y$, and is the minimum of $P(X \neq Y)$ under couplings of $X$ and $Y$. We require the following fact, which appears in \cite{kp20}.
\begin{prop}
\label{prop:couplingprob}
For $n > 0$, and $c \geq -n$ integers, $p \in [0,1]$, $X \sim \Bin(n + c, p)$ and $Y\sim \Poisson(np)$,
\[ \delta(X, Y) \leq |c|p + O(p).\]
\end{prop}
Here $\Bin(\cdot,\cdot)$ and $\Poisson(\cdot)$ are the Binomial and Poisson distributions, respectively.

\section{Breadth-first trees}
\label{sec:coupling}

Throughout this section, we assume $d = (n-(r-1))p = \Theta(1)$, i.e., $p = \Theta(1)/n$.

An $r$-graph $T$ is a {\it tree} rooted at an $(r-1)$-set $\rho$ if it can be constructed iteratively as follows: at the first step, add a new vertex $v$, and the edge $\{v\} \cup \rho$. At each subsequent step, add a new vertex $v'$, and an edge of the form $\{v'\} \cup \sigma$, where $\sigma$ is an $(r-1)$-set contained in a previously added edge (i.e., $\sigma \in D(T)$). 
A tree is a {\it path} if at each step in the construction (other than the first) the set $\sigma$  contains the vertex added at the previous step (so $\sigma$ is a subset of the edge added at the previous step). The {\it length} of the path is the number of steps in its iterative construction.

By an {\em element} of $T$ we mean an element of $T \cup V(T) \cup D(T)$, i.e., an edge of $T$, a vertex of $T$, or an $(r-1)$-set with positive co-degree. Note that if $T$ is a tree, then there is a unique path between any two elements of $T$.

The \emph{base} of an element $a$ of $T$ is the last $(r-1)$-set preceding $a$ on the path to the root. The \emph{parent} of an element $a$ is the last edge preceding $a$ on the path to the root. The \emph{children} of an edge $e$ are the $(r-1)$-sets of $e$ other than its base, the edges that have a child $(r-1)$-set of $e$ as a base, and vertices of the form $f \setminus \sigma$ where $f$ is a child edge of $e$ and $\sigma$ is the base of $e$. 

The {\em depth} of an element $a$ of $T$ is the length of the path from $a$ to $\rho$, and the depth of the tree is the largest depth of an element of $T$.

\begin{prop}
\label{prop:finitetree}
If $T$ is a finite tree, then $\tau(T) = \nu(T)$.
\end{prop}
\begin{proof}
We prove the lemma by induction on $|T|$. Let $t$ be the depth of $T$. The statement is trivial when $t \leq 1$, so assume that $t \geq 2$. Let $e$ be an edge of depth $t$, $f$ the parent of $e$. Suppose $\sigma$ and $\gamma$ are the bases of $e$ and $f$, respectively.

Let $T_0$ be the set of edges with base $\sigma$. We decompose $T \setminus (f \cup T_0)$ into trees $T_1, \dots, T_{r-1}$, where $T_1, \dots, T_{r-2}$ are trees of depth $1$ rooted at $(r-1)$-sets (not including $\sigma$ and $\gamma$) that are children of $f$ and $T_{r-1}$ is $T \setminus \left(f \cup \bigcup_{i = 0}^{r-2} T_i\right)$. Note that if $s$ is an edge of $T_i$ and $t$ is an edge of $T_j$ with $i \neq j$, then $|s \cap t| < r-1$.
For every $1 \leq i \leq r-1$, let $M_i$ and $C_i$ be a maximum matching and maximum cover of $T_i$ respectively, and note that, by induction, $|M_i| = |C_i|$. Finally, since $M = \{ e \} \cup \bigcup_{i = 1}^{r-1} M_i$ and $C = \{ \sigma \} \cup \bigcup_{i = 1}^{r-1} C_i$ are a matching and cover, respectively, in $T$ with $|M| = |C|$, the assertion follows.
\end{proof}

For an $r$-graph $G$ with distinguished $(r-1)$-set $\rho \subseteq V(G)$, and for $\gamma \in {\mathbb N}$, let $S_\gamma(\rho)$ be the subgraph of $G$ consisting of the edges of $G$ that lie on a path of length at most $\gamma$ rooted at $\rho$, and let $S(\rho)=\bigcup_\gamma S_\gamma(\rho)$.

We say $G$ is \emph{connected} iff any there is a path between any two edges of $G$. A \emph{connected component} of $G$ is a maximal connected subgraph of $G$, and is trivial if it is a single edge. For $\sigma \in D(G)$, note that $S(\sigma)$ is the connected component containing $\sigma$.

\begin{prop}
\label{prop:treewhp}
Let $H = H_r(n, p)$ and $\gamma = o(\log n/\log\log n)$. Then, for any $(r-1)$-set $\rho \subseteq V(G)$, the probability that $S_\gamma(\rho)$ is not a tree is $n^{-1+o(1)}$. 
\end{prop}
\begin{proof}
Suppose $S := S_\gamma(\rho)$ is not a tree. Then there is a vertex $v \in V(S)$, $v \notin \rho$, with the property that that there are two distinct paths from $\rho$ to $v$. 
Suppose that $v$ is such a vertex with smallest depth. These two paths may initially agree; say they both start $u_1, u_2, \ldots, u_t$ for some $t \geq r-1$, and one of the paths continues $u^{(1)}_{t+1}, \ldots$ while the other continues $u^{(2)}_{t+1}, \ldots$, with $u^{(1)}_{t+1} \neq u^{(2)}_{t+1}$. 

By appropriately truncating both paths, we can assume that the sequences $(u^{(1)}_{t+1}, u^{(1)}_{t+2}, \ldots, v)$ and $(u^{(2)}_{t+1}, u^{(2)}_{t+2}, \ldots, v)$ have nothing in common except $v$. If the first path contains $t^{(1)}$ vertices, and the second has $t^{(2)}$, then the union of the two paths is a subhypergraph of $H=H_r(n,p)$ with $t^{(1)}+t^{(2)}-t-1$ vertices and $t^{(1)}+t^{(2)}- t - (r-1)$ edges.
Noting that $t^{(1)}+t^{(2)} \leq 2\gamma$, the probability that $H$ has a subhypergraph with that many vertices and edges, including the vertices in $\rho$, is upper bounded by
\begin{equation} \label{eq-no-dense}
\sum_{k \leq 2\gamma} \binom{n-r+1}{k-r+1}\binom{\binom{k}{r}}{k-r+2} p^{k-r+2},
\end{equation}
where here $k$ is the number of vertices. Using $p=\Theta(1/n)$ and standard binomial coefficient estimates, as long as $\gamma = o(\log n/\log\log n)$ we have that the expression in \eqref{eq-no-dense} is $o(n^{-1 + o(1)})$. 
\end{proof}

Let $T_r^d$ be the (Galton-Watson-like, possibly infinite) random tree rooted at an $(r-1)$-set $\rho$ generated as follows: At the first step, we let $\rho$ give birth to $\Poisson(d)$ many edges. At each subsequent step, each $(r-1)$-set in $D(T_r^d)$ generated at the previous step gives birth to $\Poisson(d)$ many edges, independent of other $(r-1)$-sets. Since $r$ is always fixed, we write $T^d$ instead of $T_r^d$. 

\begin{prop}
\label{prop:gwtree}
For the random tree $T^d$
\begin{enumerate}
   \item $T^d$ is finite with probability 1 iff $d \leq 1/(r-1)$
    \item The expected number of edges at depth $i$ is $((r-1)d)^i/(r-1)$.
\end{enumerate}
\end{prop}
\begin{proof}
We associate with $T^d$ the random  (ordinary) tree $U$ where $V(U) = D(T)$, and $f$ is a child of $e$ iff $e$ is the base of $f$ in $T^d$. Then $U$ is a Galton-Watson tree where the number of children of each vertex has distribution $L = (r-1)\Poisson(d)$.
That the Galton-Watson tree is finite iff $\E[L] \leq 1$ (unless $L \equiv 1$), and that the number of vertices at depth $i$ is $\E[L]^i$ are basic properties of Galton-Watson trees (see{~\cite[Propositions~5.4~and~5.5]{lp17}}). The claims follow from these properties.
\end{proof}

{\bf The breadth-first tree}. For $\rho \in D(H)$ the breadth-first tree $S^*(\rho)$ rooted at $\rho$ is obtained by processing $(r-1)$-sets in the order in which they enter the tree. Specifically, let $\prec$ be an arbitrary order on $\binom{V(H)}{r-1}$, and set $P_0 = \{\rho\}$, $E_0$ to be the set of edges containing $\rho$, $D_0 = D(E_0)$, and $Q_0 = D_0 \setminus P_0$.

At step $i \in \NN$, the $(r-1)$-set $\sigma_i \in Q_{i-1}$ is {\em processed}, producing a sequence $(E_i, P_i, Q_i, X_i)$. We process $(r-1)$-sets in $Q_{i-1}$ in the order they enter $Q_{i-1}$, breaking ties according to the ordering $\prec$. Each $E_i$ will be the set of edges of a tree, with $D_i = D(E_i) = P_i \sqcup Q_i$, and $X_i = V \setminus V(E_i)$. $P_i$ is the collection of $(r-1)$-sets that have been processed, and the $(r-1)$-sets in $Q_i$ are {\em in queue}, i.e., waiting to be processed. We stop when $Q_i$ is empty, producing the tree $S^*(\rho)$.

Processing an $(r-1)$-set $\sigma_i \in Q_{i-1}$ means that we produce $E_i$ by adding to $E_{i-1}$ all edges of the form $\sigma_i \cup \{ v \}$ where $v \in X_{i-1}$. We then set $P_i = P_{i-1} \cup \{\sigma_i\}$, and set $D_i, Q_i, X_i$ as above.

Note that, for each $i$, $E_i$ is indeed a tree. This is clearly true for $E_0$. That it is true for $E_i$ when $i > 0$, follows from the fact that if $\sigma \in Q_i$, then it is in a unique edge of $E_i$, and if $\sigma \in P_i$, i.e., $\sigma$ has been processed, then $N(\sigma) \cap X_i = \emptyset$.

Let $S^*_\gamma(\rho)$ be the breath-first tree consisting of edges that lie on paths of length at most $\gamma$ rooted at $\rho$.
The following is straightforward from the definitions.
\begin{prop}
\label{prop:bftreeequality}
For any $\gamma$ and $\rho \in D(H)$, if $S_\gamma(\rho)$ is a tree then $S_\gamma(\rho) = S^*_\gamma(\rho)$.
\end{prop}

\begin{prop}
\label{prop:coupling}
For $\gamma =o(\log n / \log \log n)$ and $\rho \in D(H)$, we may couple $S^*_\gamma(\rho)$ and $T^d_\gamma$ so they are equal w.h.p.
\end{prop}
\begin{proof}
We generate $S^* := S^*(\rho)$ and $T^d$ in parallel by exposing edges as needed. When processing an $(r-1)$-set $\sigma$ in $S^*$, we also specify the number of edges containing $\sigma$ in $T^d$, coupling so these numbers agree as often as possible. If the numbers agree, then we can couple so the trees agree as well. If at any step, the trees do not agree, we say the coupling has failed. 

At the first step, we expose $N(\rho)$, specifying the edges that contain $\rho$. Note that $N(\rho)$ has distribution $\Bin(n - (r-1), p)$. By Proposition~\ref{prop:couplingprob} on processing $\rho$ the probability that the coupling fails is bounded above by \[ \delta(\Bin(n - (r-1), p), \Poisson(d)) = O(p).\]

Let $Q = \{ |T^d_\gamma| > \kappa\}$ with $\kappa = n^{0.1}$. By Proposition~\ref{prop:gwtree} and Markov's inequality $P(Q) = o(1)$.  At step $i$, we expose edges of the form $\sigma_i \cup \{v\}$ where $v \in X_i$. The distribution of the number of such  edges is $\Bin(|X_i|, p)$. By Proposition~\ref{prop:couplingprob}, the probability that the coupling fails is  bounded above by 
\[ \delta(\Bin(|X_i|, p),  \Poisson(d)) \leq \left|(n - |X_i|) \right|p + O(p) = O(p|V(E_i)|). \]
As long as $|V_i| < \kappa$, this probability is $O(n^{-0.9})$. Thus, the total probability that the coupling fails is
\[P(Q) + O(p) + \kappa O(n^{-0.9}) = o(1).\]
\end{proof}

Propositions~\ref{prop:treewhp}, \ref{prop:bftreeequality}, and \ref{prop:coupling} give what is the main point of this section.
\begin{cor}
\label{cor:coupling}
For $\gamma = o(\log n / \log \log n)$ and $\rho \in D(H)$, we may couple $S_\gamma(\rho)$ and $T^d_\gamma$ so they are equal w.h.p.
\end{cor}

\section{Proof of Theorem~\ref{thm:small-p}}
\label{sec:small-p}

Recall that the goal here is to show that for every $r \geq 3$, if $d \leq 1/(r-1)$, then w.h.p.  $\nu(H) \sim \tau(H)$, where (again recall) $H=H_r(n,p)$ is the random $r$-graph on $n$ vertices with edges selected independently with probability $p$. 

Assume first that $d = \Omega(1)$, and that $d \leq 1/(r-1)$. For any $\rho \in D(H)$, under the coupling of Section~\ref{sec:coupling}, w.h.p. $S(\rho) = T^d$. Indeed, with $\gamma =o(\log n / \log \log n)$, we have that $S_\gamma(\rho) = T_\gamma^d$ (by Corollary~\ref{cor:coupling}), and that $T^d$ is a finite tree (by Proposition~\ref{prop:gwtree}) each of which happens w.h.p. It follows that the expected number of $(r-1)$-sets of $V(H)$ in connected components that are not trees is $o(n^{r-1})$, so, by Markov's inequality, the actual number is $o(n^{r-1})$ w.h.p. 

Now note that $\tau(H) = \sum_{H'} \tau(H')$ and $\nu(H) = \sum_{H'} \nu(H')$ where the $H'$ ranges over all connected components of $H$. By Proposition~\ref{prop:finitetree}, we have
\[ \tau(H) = \sum_{H''} \tau(H'') + o(n^{r-1}) = \sum_{H''} \nu(H'') + o(n^{r-1}) = \nu(H) + o(n^{r-1}), \]
where $H''$ ranges over connected components that are trees. It suffices to show that w.h.p. $\nu(H) = \Omega(n^{r-1})$ which, by~\eqref{eq:azuma1}, follows from the fact that $\E \nu(H) = \Omega(n^{r-1})$. Now $\E \nu(H)$ is at least the expected number of {\em isolated edges} which is \[ \binom{n}{r} p \left((1 - p)^{n-r}\right)^r \geq \binom{n}{r} p \left(1 - p\right)^{rn} = \Omega(n^{r-1}). \]

Suppose now that $d = o(1)$. Let $X$ be the number of edges in $H$, $X'$ be the number of edges that share $r-1$ vertices with another edge, and $Y$ be the number of $(r-1)$-sets that lie in exactly one edge. It suffices to show that w.h.p. $X' = o(X)$. Indeed, this gives $\nu(H) \geq X + o(X)$, which combined with $\nu(H) \leq \tau(H) \leq X$ implies the assertion.

The expected number of pairs of edges that share a given $(r-1)$-set is $O(n^{2}p^2)$, hence, the expected number of pairs of edges that share an $(r-1)$-set is $O(n^{r+1}p^2)$. It follows by Markov's inequality that, for $p =o(n^{-(r+1)/2})$, w.h.p. $X' = 0$.

Otherwise, we have $X' \leq r X - Y$. Now $\E X = \binom{n}{r}p \sim \frac{n^rp}{r!}$, and 
$$
\E Y = \binom{n}{r-1} (n - (r-1)) p (1-p)^{n-r} \sim \frac{n^rp}{(r-1)!} \sim r\E X.
$$ 
This gives $\E X' \leq r\E X - \E Y = o(\E X)$ implying that  w.h.p. $X' = o(\E X)$. But for $p = \omega(n^{-r})$ (recall that we are assuming $p = \Omega(n^{-(r+1)/2})$) w.h.p. $X \sim \E X$, giving $X' = o(X)$  w.h.p.

\section{Proof of Theorem~\ref{thm:matching-large-p}}
\label{sec:matching-large-p}

Recall that the goal here is to show that for every $r \geq 3$, if $d = \omega(1)$ (where $d=(n-(r-1))p$ is the expected number of edges containing a given $(r-1)$-set), then w.h.p. $\nu(H) \sim \frac{1}{r} \binom{n}{r-1} \sim \frac{n^{r-1}}{r!}$. That $\nu \leq \frac{1}{r}\binom{n}{r-1}$ follows from the facts that every edge contains $r$ $(r-1)$-sets and each $(r-1)$-set is in at most one edge of an $(r-1)$-matching.

For the lower bound, we will use the following fractional variant of Pippenger's Theorem (see e.g. \cite{furedi1996matchingscoverings} or \cite[Theorem 4.7.1]{alonspencer4}), that is a special case of \cite[Theorem 1.5]{kahn96pippenger}. For a hypergraph~$G$ and $\varphi: G \rightarrow [0, 1]$, let $\alpha(\varphi) = \max \sum \{\varphi(e): e \in G,~ x, y \in e\}$, where the maximum is taken over all pairs of distinct vertices $x, y \in V(G)$. Recall that $\varphi$ is a fractional matching if for all $v \in V(G)$, $\sum\{\varphi(e): e \ni v\} \leq 1$.
\begin{thm}
\label{thm:pippenger}
For fixed $r$, if $G$ is $r$-uniform and $\varphi: G \rightarrow [0, 1]$ is a fractional matching then \[\nu(G) > (1 + o(1)) \sum_{A \in G}\varphi(A), \]
where $o(1) \rightarrow 0$ as $\alpha(\varphi) \rightarrow 0$.
\end{thm}
We will apply Theorem~\ref{thm:pippenger} with $G = H^{(r-1)}$; recall from the introduction that $\nu^{(r-1)}(H)=\nu(H^{(r-1)})~(=\nu(G))$.

Choose $\varepsilon$ such that $\varepsilon = o(1)$ and $\varepsilon = \omega(d^{-1/2})$ (recall that $d = \omega(1)$). Set $D = (1 + \varepsilon) d$. 
Say that a vertex $\sigma \in V(G)$ is {\it heavy} if it lies in at least $D$ edges of $G$
and define the fractional matching $\varphi: G \rightarrow [0, 1]$ via
\[ \varphi(f) =
\begin{cases}
1/D &\mbox{if } f \mbox{ contains no heavy vertices} \\
0 & \mbox{otherwise }
\end{cases} \]
It is easy to see that $\varphi$ is indeed a fractional matching; moreover, any two distinct vertices $\sigma_1, \sigma_2 \in V(G)$ are contained in at most one edge of $G$, implying that $\alpha(\varphi) \leq 1/D \rightarrow 0$ (as $n \rightarrow \infty$). So, recalling Theorem~\ref{thm:pippenger}, to complete the proof that $\nu \gtrsim \frac{n^{r-1}}{r!}$ it suffices to show 
\begin{equation} \label{using-pippenger}
\sum_{A \in G} \varphi(A) \sim \frac{n^{r-1}}{r!} \mbox{\quad w.h.p.}
\end{equation}

We will establish (\ref{using-pippenger}) by showing that the expected number of edges in $G$ that contain heavy vertices is $o(n^rp)$, and so (by Markov's inequality) is w.h.p. $o(n^rp)$. Now note that the number of edges in $G$ (which equals the number of edges in $H$) is ${\rm Bin}(\binom{n}{r},p)$, so w.h.p. $|G| \sim \binom{n}{r} p$, and that also (by our choice of $\varepsilon$) $D \sim np$.
It follows that w.h.p. \[ \sum_{A \in H'}\varphi(A) \sim \binom{n}{r}p \cdot \frac{1}{np} \sim \frac{n^{r-1}}{r!}, \]
verifying (\ref{using-pippenger}).

So, letting $X$ be the number of edges in $G$ that contain heavy vertices, it remains to show $E(X)=o(n^rp)$. For each $\ell \in {\mathbb N}$, let $X_\ell$ be the number of vertices in $G$ that are in exactly $\ell$ edges, and for $a < b \in {\mathbb R}$ let $X_{[a,b)}$ be the number of vertices in $G$ that are in at least $a$ and fewer than $b$ edges. For a vertex $\sigma \in V(G)$, let $n_\sigma$ be the number of edges that $\sigma$ is in. We have
\begin{eqnarray}
\E X & = & \sum_{\gamma \geq \varepsilon:~ (1+\gamma)d \in {\mathbb N}}  \E X_{(1+\gamma)d} (1+\gamma)d \nonumber \\
& \leq & \sum_{i \geq 0} \E X_{[(1+2^i\varepsilon)d,(1+2^{i+1}\varepsilon)d)} (1+2^{i+1}\varepsilon)d \nonumber \\
& = & \sum_{i \geq 0} \binom{n}{r-1} P(n_\sigma \in [(1+2^i\varepsilon)d,(1+2^{i+1}\varepsilon)d)) (1+2^{i+1}\varepsilon)d \nonumber \\
& \leq & \binom{n}{r-1}d\sum_{i \geq 0} P(n_\sigma \geq (1+2^i\varepsilon)d)(1+2^{i+1}\varepsilon) \nonumber \\
& \lesssim & \frac{n^rp}{(r-1)!} \sum_{i \geq 0} P(n_\sigma \geq (1+2^i\varepsilon)d)(1+2^{i+1}\varepsilon).
\end{eqnarray}
Now, noting that $n_\sigma$ is ${\rm Bin}(n-(r-1),p)$, Theorem~\ref{thm:chernoff} gives
\begin{eqnarray*}
P(n_\sigma \geq (1+2^i\varepsilon)d) & \leq & \exp\left[-(2^i\varepsilon d)^2/(2(d + 2^i\varepsilon d/3))\right]\\
& = & \frac{1}{\exp\left[2^{2i}\varepsilon^2d/(2(1 + 2^i\varepsilon/3))\right]},
\end{eqnarray*}
so that
\begin{eqnarray*}
E(X)  \lesssim n^rp\sum_{i \geq 0} \frac{1 + 2^{i+1}\varepsilon}{\exp\left[2^{2i}\varepsilon^2d/(2(1 + 2^i\varepsilon/3))\right]}
 \leq n^rp\sum_{i \geq 0} \frac{1 + 2^{i+1}\varepsilon}{\exp[2^i \varepsilon d/4]}.
\end{eqnarray*}
Finally, recalling that $\varepsilon = o(1)$ and $\varepsilon = \omega(d^{-1/2})$ we obtain $E(X) =  o(n^rp)$ as required. 

\section{Proof of Theorem~\ref{thm:matching-medium-p}}
\label{sec:matching-medium-p}

Recall that the goal here is to prove that for every $r \geq 3$, if $d = \Theta(1)$ then w.h.p.
\[ \nu(H) > (1 + o(1)) \frac{1}{r} \alpha_r(d) \binom{n}{r-1}\]
where
\[\quad \alpha_r(d) = 1 - \left(\frac{1}{(r-1)d + 1}\right)^{1/(r-1)}.\] 

Given an $r$-graph $G$ and a function $w: G \rightarrow [0, 1]$ (referred to as weights on the edges), the greedy $(r-1)$-matching corresponding to $w$ is denoted by $M_w$. That is, we consider edges in increasing order of weights, and add an edge to the matching iff it shares at most $(r-2)$ vertices with any edge already in the matching. When $w$ is uniform from $[0, 1]^{|G|}$, we denote by $M^*$ the random greedy $(r-1)$-matching. Note that we may assume that weights are distinct, since this happens with probability 1 when $w$ is chosen uniformly from $[0, 1]^{|G|}$ .

We will consider the random greedy $(r-1)$-matching on $H$. Given an $(r-1)$-set $\rho \in \binom{V}{r-1}$, we would like to bound $P(\rho \in D(M^*))$, where the probability is over choices of $H$ and $w$. Since $d_{M^*}(\sigma) \leq 1$ for each $\sigma \in D(M^*)$, we obtain \[ \E|M^*| = \frac{1}{r} P(\rho \in D(M^*)) {\binom{n}{r-1}}, \]
which, by~\eqref{eq:azuma1}, implies that w.h.p. \[ \nu > |M^*| = (1 + o(1)) \frac{1}{r} P(\rho \in D(M^*)) \binom{n}{r-1}. \]
Thus, it suffices to show that 
\begin{equation}
\label{eq:matchingmain}
P(\rho \in D(M^*)) = 1 - \left(\frac{1}{(r-1)d + 1}\right)^{1/(r-1)}.
\end{equation}
\begin{remark}
\label{rmk:conditionalprob}
The proof of \eqref{eq:matchingmain} is an immediate generalization the proof of \cite[equation (11)]{kp20}. However, in \cite{kp20}, the existence of a triangle in a random 2-graph containing the pair of vertices $\{x,y\}$ depends on the existence of the edge $xy$, hence the corresponding equation bounds a conditional probability. For us, each $r$-set of vertices is an edge of the random $r$-graph $H$ independently of other $r$-sets, so the probability in \eqref{eq:matchingmain} is not a conditional probability.
\end{remark}

Before proving~\eqref{eq:matchingmain}, we need some notation and definitions. Note that the following proof is a straightforward generalization of the ideas in \cite{kp20}, and is based on the coupling from Section~\ref{sec:coupling}.

For a finite tree $T$, we work with the following recursive survival rule for $(r-1)$-sets, in which we evaluate $(r-1)$-sets in any order for which each $(r-1)$-set appears earlier than its base (further specification of the order doesn't affect the outcome), and ``dies'' means fails to survive:
\begin{equation}
\mbox{An $(r-1)$-set  dies iff it is the base of an edge all of whose other $(r - 1)$-sets survive.}
\end{equation}
For example, any $(r-1)$-set that is the base of no edges survives.

For an $r$-graph $G$, $\rho \in D(G)$ and a weight function $w$ uniform in $[0, 1]^{|G|}$, let \[ W(\rho) = \{\sigma \in D(G): \mbox{there is a path from $\rho$ to $\sigma$ on which the weights decrease}\}. \]
Note that $W(\rho)$ is a random subgraph of $S(\rho)$. The main point here is that if $W(\rho)$ is a finite tree, then $\rho$ is in $D(M^*)$ iff it dies when we apply the survival rule to $W(\rho)$. To see this, observe that an $(r-1)$-set $\sigma$ dies iff it is the base of an edge that is in $M$.

When $H = T^d$ with root $\rho$, we write $W^d$ for $W(\rho)$. When talking about $W^d$, the probabilities are over choices of $H$ and $w$. We have
\begin{prop}
\label{prop:matching-finite}
$W^d$ is finite with probability $1$.
\end{prop}
\begin{proof}
By Proposition~\ref{prop:gwtree} (part 2), the expected number of $(r-1)$-sets at depth $i$ is $((r-1)d)^i/((r-1)i!)$. As $i \rightarrow \infty$ this goes to 0, implying the assertion.
\end{proof}
By Corollary~\ref{cor:coupling} and Proposition~\ref{prop:matching-finite}, for $\rho \in \binom{V}{r-1}$, we may couple $W(\rho)$ and $W^d$ to agree w.h.p. It follows that $P(\rho \in D(M^*))$ tends to the probability that the root dies in $W^d$. Hence, to complete the proof, it suffices to show the following.
\begin{prop}\label{survival}
\label{prop:matching-prob}
Under the survival rule, the root of $W^d$ survives with probability \[ ((r-1)d + 1)^{-1/(r-1)}. \]
\end{prop}
\begin{proof}
It will be convenient to extend $w$ to $(r-1)$-sets: set $w(\rho) = 1$, and for any other $(r-1)$-set $\sigma \in D(W^d)$, let $w(\sigma)$ be the
weight of the (unique) edge containing $\sigma$ of minimum depth.

Let $f(x)$ be the probability that an $(r-1)$-set of weight $x$ survives. Trivially $f(0) = 1$.  The probability that an edge with weight $x$ survives is
\begin{align*}
    f(x) = \sum_{k} \frac{e^{-d}d^k}{k!}\left[1 - \int_0^{x} f^{r-1}(y)\,dy\right]^k.
\end{align*}
To see this note that for a given child edge of a $(r-1)$-set $\sigma$, the probability that at least one other $(r-1)$-set dies  is $1 - \int_0^x f^{r-1}(y)dy$. Given that there are $k$ edges on  $\sigma$, the probability that for all of them at least one $(r-1)$-set dies is $\left[1 - \int_0^x f^{r-1}(y)\,dy\right]^k$. Let $F(x) = \int_0^{x} f^{r-1}(y)\,dy$, so $F(0) = 0$ and
\begin{align*}
F'(x) = f^{r-1}(x) & = \left(\sum_{k} \frac{e^{-d}d^k}{k!}\left[1 - \int_0^{x} f^{r-1}(y)\,dy\right]^k\right)^{r-1}\\
& = \left(e^{-d}\left[\sum_{k} \frac{d^k}{k!}\left(1 - F(x)\right)^k\right]\right)^{r-1} \\
& = e^{-(r-1)dF(x)}.
\end{align*}
Solving this gives \[F(x) = \frac{1}{(r-1)d}\ln\left((r-1)dx + 1\right). \]
Now, we have \[ F'(x) = f^{r-1}(x) =  \frac{1}{(r-1)dx + 1}, \] which implies
\[ f(x) = ((r-1)dx + 1)^{-1/(r-1)}. \]
\end{proof}

\section{Proofs of Theorems~\ref{thm:cover-absolute}, \ref{thm:cover-medium-p-small-r}, and \ref{thm:cover-medium-p-large-r}}
\label{sec:covers}

In each of the following subsections, we show the existence of an $(r-1)$-cover for an arbitrary $r$-graph $G$, which gives the bounds in Theorem~\ref{thm:cover-absolute}. These covers are based on constructions for the hypergraph Tur\'an problem  (which give $(r-1)$-covers for the complete $r$-graph $K_n^r$).
We then proceed to show that, for the random hypergraph $H$ with $d = \Theta(1)$, these covers can be improved upon, resulting in the bounds in Theorems~\ref{thm:cover-medium-p-small-r} and \ref{thm:cover-medium-p-large-r}.

For $l \in \NN$, by a \emph{uniformly random partition} of $V(H)$ into $l$ blocks, we mean the partition $(V_0, V_1 , \dots, V_{l-1})$  where every $v \in V(H)$ is placed into each of $V_0, \dots, V_{l-1}$ with probability $1/l$, independent of other vertices. We say that an $m$-set $A$ of $[n]$ has {\it ordered type} $\overline{a}=a_0a_1\cdots a_{l-1}$ (with respect to this partition) if $|A\cap V_i|=a_i$ for each $i$. So $(a_0,\ldots,a_{l-1})$ forms a weak composition of $m$ into $l$ parts. $A$ has \emph{unordered type} $\overline{a} = a_0 \geq a_1 \geq \cdots \geq a_{l-1}$ if the ordered type of $A$ is a permutation of $a_0a_1\cdots a_{l-1}$.

\subsection{Covers for 3-graphs}
\label{sec:cover3}

The following is the standard construction showing that every $2$-graph $G$ has a bipartite (and, hence, triangle-free) subgraph with at least half the edges. 
Let $G$ be $3$-graph, and $(V_0, V_1)$ be a uniformly random partition of $V(G)$. Set
\[ C_0 = \{\sigma \in \binom{V}{2}: \sigma \mbox{ has unordered type 20}\}. \]
Let $C = C_0 \cap D(G)$, and notice that $C$ is a 2-cover of $G$. Note that $P(\sigma \in C_0) = 1/2$, so $\E|C| = |D(G)|/2$. Hence, by the first moment method, we have
\[\tau(G) \leq \frac{1}{2}|D(G)|. \]

{\bf Improved bounds for $d = \Theta(1)$: } We now improve the cover $C_0$ for $H$ with $d = \Theta(1)$ by removing some additional elements (this construction appeared in \cite{kp20}).
Let
$$
C_1 = \{\sigma \in C_0 : \mbox{all edges containing $\sigma$ are contained in the same block as $\sigma$}\}.
$$
Note that $C_0 \setminus C_1$ may not be a $2$-cover of $H$; indeed, if $e \in H$ has unordered type $30$, and has the property that all three pairs in $\binom{e}{2}$ are in $C_1$, then $e$ is not covered by $C_0\setminus C_1$. However, setting
$$
C_2 = \{ \sigma \in C_1 : \mbox{there is an edge $e \supset \sigma$ such that for every $\sigma' \in \binom{e}{2}$ we have $\sigma' \in C_1$} \}
$$
we obtain 
$$
C = C(V_0, V_1) := (C_0 \setminus C_1) \cup C_2
$$
is again a $2$-cover of $H$.

We will show that, for each~$\sigma \in \binom{V}{2}$, 
\begin{equation}
\label{eq:r3probability}
P(\sigma \in C) \rightarrow \frac{1}{2} \left[1 - \exp\left(-\frac{d}{2}\left(1 + e^{-d}\right)\right) \right] = \beta_3(d).
\end{equation}
The proof of \eqref{eq:r3probability} is an immediate generalization the proof of \cite[equation (15)]{kp20}. For the same reason as in Remark~\ref{rmk:conditionalprob}, the probability in \cite[equation (15)]{kp20} needs to be conditional, whereas the probability in \eqref{eq:r3probability} is unconditional.

From this it follows that $\E|C| = \binom{n}{2}\beta_3(d)$, which, by \eqref{eq:azuma1}, implies the asserted bound.

The number of edges containing $\sigma$ has distribution $\Bin(n-2, p) \rightarrow \Poisson(d)$ implying
\begin{equation}
\label{eq:r3c1}
P(\sigma \in C_1) \sim \frac{1}{2} \sum_{k \geq 0} \frac{e^{-d}{d^k}}{k!}\frac{1}{2^k} = \frac{1}{2} e^{-d/2}.
\end{equation}
To estimate $P(\sigma \in C_2)$, we rely on the coupling from Section~\ref{sec:coupling} (with $\gamma = 2$). Given an edge $e \in V_0$ containing $\sigma \in C_1$, the probability that all edges with base in $\binom{e}{2} \setminus \{ \sigma \}$ are contained in $V_0$ is
\[ \sum_{l, m \geq 0} \frac{e^{-d} d^l}{l!} \frac{e^{-d} d^m}{m!} \frac{1}{2^m}\frac{1}{2^l}
= \left(\sum_{l\geq 0} \frac{e^{-d} d^l}{l!} \frac{1}{2^l}\right)^2
= e^{-d}. \]
It follows that
\begin{equation}
\label{eq:r3c2}
P(\sigma \in C_2) \sim \frac{1}{2} \sum_{k \geq 0} \frac{e^{-d}d^k}{k!}\frac{1}{2^k} \left(1 - \left(1 - e^{-d}\right)^k\right) = \frac{1}{2} \left[ e^{-d/2} - \exp\left(-\frac{d}{2}\left(1 + e^{-d}\right)\right)\right].
\end{equation}
Finally, combining \eqref{eq:r3c1} and \eqref{eq:r3c2} gives \eqref{eq:r3probability}.

\subsection{Covers for 4-graphs}
\label{sec:cover4}

The following is based on a construction of Tur\'an~\cite{turan41}.
Let $G$ be $4$-graph, and $(V_0, V_1, V_2)$ a uniformly random partition of $V(G)$. Set
\begin{itemize}
    \item $C_0^1 = \left\{ \sigma \in \binom{V}{3}: \sigma \mbox{ has unordered type 300} \right\}$,
    \item $C_0^2 = \left\{ \sigma \in \binom{V}{3}: \sigma \mbox{ has ordered type in } \{210, 021, 102\} \right\}$.
\end{itemize}
Note that $C = C(V_0, V_1, V_2) = (C_0^1 \cup C_0^2) \cap D(G)$ is a 3-cover of $G$. We have,  $P(\sigma \in C_0^1) = 1/9$ and $P(\sigma \in C_0^2) = 1/3$, giving
$\E|C| = 4|D(G)|/9$. The first moment method then gives
\[\tau(G) \leq \frac{4}{9}|D(G)|. \]

{\bf Improved bounds for $d = \Theta(1)$: } We now improve the cover $C$ for $H$ with $d = \Theta(1)$.
Let $C = C(V_0, V_1, V_2) = \left( (C_0^1 \setminus C_1^1) \cup C_2^1 \right) \bigcup \left(C_0^2 \setminus C_1^2\right)$, where
\begin{itemize}
    \item $C_1^1 = \{ \sigma \in C_0^1 : \mbox{all edges containing $\sigma$ are contained in the same block as $\sigma$}\}$,
    \item $C_2^1 = \{ \sigma \in C_1^1 : \exists~ \mbox{edge $e \supset \sigma$ such that $\forall$ $\sigma' \in \binom{e}{3}$, we have  $\sigma' \in C_1^1$}\}$,
    \item $C_1^2 = \{ \sigma \in C_0^2 : \mbox{all edges containing $\sigma$ have unordered type 310} \}$.
\end{itemize}

It is easy to verify that this is indeed a $3$-cover for $H$. Indeed, edges of unordered type $400$ and $310$ are covered by $(r-1)$-sets in $ \left( (C_0^1 \setminus C_1^1) \cup C_2^1 \right)$;  edges of unordered type $220$ and $211$ are covered by $(r-1)$-sets in $\left(C_0^2 \setminus C_1^2\right)$. The asserted bound follows, by \eqref{eq:azuma1}, from the following: For each $\sigma \in \binom{V}{3}$
\begin{equation}
\label{eq:r4probability}
P(\sigma \in C) \lesssim \frac{4}{9} \left[1 - \exp\left(-\frac{d}{3}\left(2 + e^{-2d}\right)\right)\right] = \beta_4(d).
\end{equation}

The number of edges containing $\sigma$ has distribution $\Bin(n-3, p) \rightarrow \Poisson(d)$ implying
\begin{equation}
\label{eq:r4c11}
P(\sigma \in C_1^1) \sim \frac{1}{9} \sum_{k \geq 0} \frac{e^{-d}{d^k}}{k!}\frac{1}{3^k} = \frac{1}{9} e^{-2d/3}.
\end{equation}
To estimate $P(\sigma \in C_2^1)$, we rely on the coupling from Section~\ref{sec:coupling}. Given an edge $e \in V_0$ containing $\sigma \in C_1^1$, the probability that all edges with base in $\binom{e}{3} \setminus \sigma$ are contained in $V_0$ is
\[ \left(\sum_{l\geq 0} \frac{e^{-d} d^l}{l!} \frac{1}{3^l}\right)^3
= e^{-2d}. \]
It follows that
\begin{equation}
\label{eq:r4c21}
P(\sigma \in C_2^1) \sim \frac{1}{9} \sum_{k \geq 0} \frac{e^{-d}d^k}{k!}\frac{1}{3^k} \left(1 - \left(1 - e^{-2d}\right)^k\right) = \frac{1}{9} \left[ e^{-2d/3} - \exp\left(-\frac{d}{3}\left(2 + e^{-2d}\right)\right)\right].
\end{equation}
Combining \eqref{eq:r4c11} and \eqref{eq:r4c21} gives
\begin{equation}
\label{eq:r4c1}
P(\sigma \in \left( (C_0^1 \setminus C_1^1) \cup C_2^1 \right)) \sim \frac{1}{9}\left[1 - \exp\left(-\frac{d}{3}\left(2 + e^{-2d}\right)\right)\right]
\end{equation}

Similarly, we have $P(\sigma \in C_1^2) \sim e^{-2d/3}/3$  which immediately gives
\begin{equation}
\label{eq:r4c2}
P(\sigma \in \left(C_0^2 \setminus C_1^2\right)) \sim \frac{1}{3}\left[1 - e^{-2d/3}\right] \leq \frac{1}{3}\left[1 - \exp\left(-\frac{d}{3}\left(2 + e^{-2d}\right)\right)\right]
\end{equation}
Note that the second inequality is easy (and, strictly speaking, unnecessary) but helps to simplify the analysis.
Finally, combining \eqref{eq:r4c2} and \eqref{eq:r4c1} gives  \eqref{eq:r4probability}, implying the assertion.

\subsection{Covers for 5-graphs}
\label{sec:cover5}

We adapt a construction of Giraud~\cite{giraud1990}. Let $G$ be a $5$-graph, $(V_0, V_1)$ be a uniformly random partition of $V(G)$, and let $f : V_0 \times V_1 \rightarrow \{0, 1\}$ be a uniformly random function. That is, for each $(x, y) \in V_0 \times V_1$, $f(x, y)$ is $0$ or $1$ with probability of $1/2$ independent of other inputs.

For $A_0 \subseteq V_0$ and $A_1 \subseteq V_1$, we set $f(A_0, A_1) = \sum_{(x, y)\in A_0 \times A_1} f(x, y)$. For a set $A \subseteq V$, we let $f(A) = f(A \cap V_0, A \cap V_1)$. Given $\sigma \in \binom{V(H)}{4}$ of type 22 (with respect to $V_0, V_1$), we say $\sigma$ is \emph{even} if $f(\sigma)$ is even. Note that
\[ \mbox{if $\sigma \cap V_0 = \{x, y\}$, then $\sigma$ is even iff $f(x, \sigma \cap V_1) \equiv f(y, \sigma \cap V_1) \pmod 2$}, \]
and
\[ \mbox{if $\sigma \cap V_1 = \{u, v\}$, then $\sigma$ is even iff $f(\sigma \cap V_0, u) \equiv f(\sigma \cap V_0, v) \pmod 2$}. \]

Let $C = C(V_0, V_1) = (C_0^1 \cup C_0^2) \cap D(G)$, where
\begin{itemize}
    \item $C_0^1 = \left\{ \sigma \in \binom{V}{4}: \sigma \mbox{ has unordered type 40} \right\}$,
    \item $C_0^2 = \left\{ \sigma \in \binom{V}{4}: \mbox{ $\sigma$ has type 22 and $\sigma$ is even} \right\}$.
\end{itemize}

We claim that $C$ is a $4$-cover of $G$. Indeed, any edge in $G$ that has at least four vertices in the same block is evidently covered by something in $C$. For an edge $e= \{ v_1, v_2, v_3, v_4, v_5\}$ of $G$ that is not covered by this case, assume (without loss of generality) that $v_1$ and $v_2$ are in $V_0$ and $v_3, v_4$ and $v_5$ are in $V_1$. Now, at least two of $f(\{v_1, v_2\}, v_3)$, $f(\{v_1, v_2\}, v_4)$ and $f(\{v_1, v_2\}, v_5)$ must have the same parity, and, hence, the corresponding 4-set must be in $C$.

We now compute the expected size of $C$. For $\sigma \in \binom{V}{4}$, note that $P(\sigma \in C_0^1) = 1/8$. Note that $P(\sigma\mbox{ has type } 22) = 3/8$ and $P(\sigma\mbox{ is even}) = 1/2$, implying $P(\sigma \in C_0^2) = 3/16$.
It follows that $\E|C| = 5|D(G)|/16$. The asserted bound follows from the first-moment method.

{\bf Improved bounds for $d = \Theta(1)$: }
We now improve the cover $C_0$ for $H$ with $d = \Theta(1)$.

Suppose $\sigma$ is even, and let $e$ be an edge of the form $e = \sigma \cup \{v\}$. If $v \in V_0$, we say $e$ is \emph{compatible} with $\sigma$ if $f(v, \sigma \cap V_1)$ has the same parity as $f(x, \sigma \cap V_1)$, for any $x \in \sigma \cap V_0$. Similarly, if $v \in V_1$, we say $e$ is \emph{compatible} with $\sigma$ if $f(\sigma \cap V_0, v)$ has the same parity as $f(\sigma \cap V_0, x)$, for any $x \in \sigma \cap V_1$. The motivation behind this definition is that if $e$ is compatible with $\sigma$, then it is covered by multiple 4-sets, which we exploit in the following construction.

Let $C = C(V_0, V_1) = \left( (C_0^1 \setminus C_1^1) \cup C_2^1 \right) \bigcup \left( (C_0^2 \setminus C_1^2) \cup C_2^2 \right)$, where
\begin{itemize}
    \item $C_1^1 = \{ \sigma \in C_0^1 : \mbox{all edges containing $\sigma$ are contained in the same block as $\sigma$} \}$,
    \item $C_2^1 = \{ \sigma \in C_1^1 : \exists~ \mbox{edge $e \supset \sigma$ such that $\forall$ $\sigma' \in \binom{e}{4}$, we have  $\sigma' \in C_1^1$}\}$,
    \item $C_1^2 = \{ \sigma \in C_0^2 : \mbox{every edge containing  $\sigma$ is compatible with $\sigma$} \}$,
    \item $C_2^2 = \{ \sigma \in C_1^2 : \exists~ \mbox{edge $e \supset \sigma$ such that $\forall$ $\sigma' \in \binom{e}{4}$ of type 22, we have  $\sigma' \in C_1^2$}\}$.
\end{itemize}
An argument similar to those in the preceding sections suffices to show that $C$ is a $4$-cover. The asserted bound follows from the fact that, for each $\sigma \in \binom{V}{4}$, 
\begin{equation}
\label{eq:r5probability}
P(\sigma \in C) \lesssim \frac{5}{16} \left[1 - \exp\left(-\frac{d}{2}\left(1 + e^{-d}\right)\right) \right],
\end{equation}
where the probability is over choices of the partition $(V_0, V_1)$ and the function $f$.

Note first that $P(\sigma \in C_0^1) = 1/8$. The number of edges containing $\sigma$ has distribution $\Bin(n-4, p) \rightarrow \Poisson(d)$ implying
\begin{equation}
\label{eq:r5c11}
P(\sigma \in C_1^1) \sim \frac{1}{8} \sum_{k \geq 0} \frac{e^{-d}{d^k}}{k!}\frac{1}{2^k} = \frac{1}{8} e^{-d/2}.
\end{equation}
To estimate $P(\sigma \in C_2^1)$, we rely on the coupling from Section~\ref{sec:coupling}. Given an edge $e \in V_0$ containing $\sigma \in C_1^1$, the probability that all edges with base in $\binom{e}{4} \setminus \sigma$ are contained in $V_0$ is
\[ \left(\sum_{l\geq 0} \frac{e^{-d} d^l}{l!} \frac{1}{2^l}\right)^4
= e^{-2d}. \]
It follows that
\begin{equation}
\label{eq:r5c21}
P(\sigma \in C_2^1) \sim \frac{1}{8} \sum_{k \geq 0} \frac{e^{-d}d^k}{k!}\frac{1}{2^k} \left(1 - \left(1 - e^{-2d}\right)^k\right) = \frac{1}{8} \left[ e^{-d/2} - \exp\left(-\frac{d}{2}\left(1 + e^{-2d}\right)\right)\right].
\end{equation}
Combining \eqref{eq:r5c11} and \eqref{eq:r5c21}, we obtain that 
\begin{align}
\label{eq:r5c1}
\nonumber P (\sigma \in \left( (C_0^1 \setminus C_1^1) \cup C_2^1 \right)) & = \frac{1}{8} \left[1 - \exp\left(-\frac{d}{2}\left(1 + e^{-2d}\right)\right) \right]\\
& \leq \frac{1}{8} \left[1 - \exp\left(-\frac{d}{2}\left(1 + e^{-d}\right)\right) \right]. 
\end{align}
The second inequality serves to simply the final bound, and, hence, the analysis.
Given $\sigma \in C_0^2$, the probability that an edge $e = \sigma \cup \{v\}$ is compatible with $\sigma$ is $1/2$. It follows that 
\begin{equation}
\label{eq:r5c12}
P(\sigma \in C_1^2) \sim \frac{3}{16} \sum_{k \geq 0} \frac{e^{-d}{d^k}}{k!}\frac{1}{2^k} = \frac{3}{16} e^{-d/2}.
\end{equation}

Let $e$ be an edge containing $\sigma \in C_1^2$. Then $e$ must be of type 32, and so there are two $4$-sets in  $\binom{e}{4} \setminus \{\sigma\}$ that are of type $22$. The probability that these two $4$-sets are also in $C_1^2$ is
\[ \sum_{l, m \geq 0} \frac{e^{-d} d^l}{l!} \frac{e^{-d} d^m}{m!} \frac{1}{2^m}\frac{1}{2^l}
= \left(\sum_{l\geq 0} \frac{e^{-d} d^l}{l!} \frac{1}{2^l}\right)^2
= e^{-d}. \]
It follows that
\begin{equation}
\label{eq:r5c22}
P(\sigma \in C_2^2) \sim \frac{3}{16} \sum_{k \geq 0} \frac{e^{-d}d^k}{k!}\frac{1}{2^k} \left(1 - \left(1 - e^{-d}\right)^k\right) = \frac{3}{16} \left[ e^{-d/2} - \exp\left(-\frac{d}{2}\left(1 + e^{-d}\right)\right)\right].
\end{equation}
Combining \eqref{eq:r5c12}  and \eqref{eq:r5c22}, we obtain that
\begin{equation}
\label{eq:r5c2}
P (\sigma \in \left( (C_0^2 \setminus C_1^2) \cup C_2^2 \right)) \sim \frac{3}{16} \left[1 - \exp\left(-\frac{d}{2}\left(1 + e^{-d}\right)\right) \right]. 
\end{equation}
Finally, \eqref{eq:r5c1} and \eqref{eq:r5c2} together give \eqref{eq:r5probability}, implying the assertion.

\subsection{Larger arity}

For $r \geq 6$, we use the following construction of Frankl and R{\"o}dl~\cite{fr85}, along with an improvement due to Sidorenko~\cite{sidorenko1997upper}.

\paragraph{The Frankl-R\"odl construction}
Let $G$ be an $r$-graph, $l$ be a positive integer, and $(V_0, \dots, V_{l-1})$ be a uniformly random partition of $V(G)$.
For $A \subseteq V(G)$, define 
\[ d(A) := \left|\left\{ i \in \{0, \dots, l-1\}: A \cap V_i = \emptyset \right\}\right|, \]
and 
\[ w(A) := \sum_{i = 0}^{l-1} i |A \cap V_i|.\]
For $0 \leq j \leq l-1$, let $\mathcal{C}_j$ be the family
\[ \mathcal{C}_j := \left\{ \sigma \in D(G): \left( w(\sigma) + j\right) \bmod l  \in \{0, \dots, d(\sigma)\} \right\}. \]
We claim that, for every $0 \leq j\leq l-1$, $\mathcal{C}_j$ is a cover of $G$.
To see this, note that, for every $e \in G$, there are $l - d(e)$ indices $i$ such that $e \cap V_i \neq \emptyset$, at least one of which must be in
\[ (w(e) + j) \bmod l,~(w(e) + j - 1) \bmod l,~\dots~,~(w(e) + j - d(e))  \bmod l. \] 
Fix $i$ to be such an index. Let $x \in e \cap V_i$, and $\sigma = e \setminus \{x\}$.  Now, since $w(\sigma) \equiv w(e) - i \pmod l$ and $d(\sigma) \geq d(e)$, we have $0 \leq (w(\sigma) + j) \bmod l \leq d(\sigma) $ implying $\sigma \in \mathcal{C}_j$.

\begin{lemma}
\label{lem:frcover}
For any $l \in \NN$ and any $r$-graph $G$ \[ \tau(G) \leq \left|D(G)\right|\left[\frac{1}{l} + \left(1 - \frac{1}{l}\right)^{r-1}\right] . \]
\end{lemma}
\begin{proof}
The assertion follows by the first moment method, along with the pigeonhole principle from \begin{equation}
\label{eq:frbound}
\E\sum_{j = 0}^{l-1}|\mathcal{C}_j| = \left[1 + l\left(1 - \frac{1}{l}\right)^{r-1}\right] \left|D(G)\right|.    
\end{equation} 

To see \eqref{eq:frbound}, note that each $\sigma \in D(G)$ belongs to exactly $d(\sigma) + 1$ of the covers $\mathcal{C}_0, \dots, \mathcal{C}_{l-1}$, giving
\[ \sum_{j = 0}^{l-1} |\mathcal{C}_j| = \sum_{\sigma \in D(G)} (d(\sigma) + 1) = |D(G)| + \sum_{i = 0}^{l-1}|\mathcal{A}_i|,\]
where $\mathcal{A}_i = \{ \sigma \in D(G) : e \cap V_i = \emptyset\}$ with
\[ \E|A_i| =  \left| D(G)\right| P(\{\sigma \cap V_i = \emptyset\}) =  \left|D(G)\right| \left(1 - \frac{1}{l}\right)^{r-1}. \]
\end{proof}

\paragraph{Sidorenko's Improvement}

Consider a function $f: V_0\times\dots\times V_{l-1} \rightarrow \{0, 1\}$. For a set $A \in \binom{V(G)}{2l}$ with $|A \cap V_i| = 2$ for every $i$, let
\begin{equation}
\label{eq:randomsum} q(A) =  \sum_{x_i \in A \cap V_i} f(x_0, \dots, x_{l-1}).
\end{equation}
Let $\mathcal{E} = \{ \sigma \in D(G): |\sigma \cap V_i| \geq 2 \mbox{ for every }i\}$, and let $\pi(\sigma)$ be the $2l$ element set obtained by taking the two maximal elements (with respect to an arbitrary linear order on $V$) from each of the sets $\sigma \cap V_i$, $0 \leq i \leq l-1$. For $0 \leq j \leq l-1$, let $\mathcal{C}'_j$ be the family
\[ \mathcal{C}'_j = \left\{ \sigma \in \mathcal{C}_j: \sigma \notin \mathcal{E} \mbox{ or } q(\pi(\sigma)) \mbox{ is even} \right\}. \]
We claim that, for any $j$, $\mathcal{C}'_j$ is a cover. Fix $0 \leq j \leq l-1$ and recall from the Frankl-R\"odl construction that, given an edge $e \in G$, there exists an index~$i$ such that $e \cap V_i \neq \emptyset$ and $\sigma = e \setminus \{ x \} \in \mathcal{C}_j$ for any $x \in e \cap V_i$.
If $e \setminus \{ x \} \notin \mathcal{E}$, then $e$ is covered by $\mathcal{C}'_j$. Otherwise, if $e \setminus \{ x \} \in \mathcal{E}$, then $|e \cap V_i| \geq 3$ and $|e \cap V_k| \geq 2$ for $k \neq i$.
Suppose $\{x, y, z\} \subseteq e \cap V_i$ are the three maximal elements, and consider the sum \[ Q := q(\pi(e \setminus \{ x \})) + q(\pi(e \setminus \{ y \})) + q(\pi(e \setminus \{ z \})). \] Observe that every term of the form $f(x_0, \dots, x_{l-1})$ appears twice in the sum implying $Q$ is even. It follows that at least one of $q(\pi(e \setminus \{x\})), q(\pi(e \setminus \{y\}))$ or $q(\pi(e \setminus \{z\}))$ is even and, hence, in $\mathcal{C}'_j$.

\begin{lemma}
\label{lem:sidorenkocover}
For any $l \leq r/2$ and any $r$-graph $G$ \[ \tau(G) \leq \frac{|D(G)|}{2}\left[ \frac{1}{l} + \left(3 + \frac{r-1}{l-1}\right)\left(1 - \frac{1}{l}\right)^{r-1}\right]. \]
\end{lemma}
\begin{proof}
Let $f: V_0\times\dots\times V_{l-1} \rightarrow \{0, 1\}$ be a random function whose entries are chosen uniformly and independently. For any $\sigma \in D(G)$,  $P(q(\pi(\sigma))\mbox{ is odd}) = 1/2$.  Hence,
\[\E|\{\sigma \in \mathcal{E}: q(\pi(\sigma))\mbox{ is odd}\}| = \frac{1}{2}|\mathcal{E}|.\]
By the first moment method, there exists a function $f^*$ such that
\begin{equation}
\label{eq:randomfunc}
    |\{\sigma \in \mathcal{E}: q(\pi(\sigma))\mbox{ is odd}\}| \geq |\mathcal{E}|/2.
\end{equation}
Choosing $f^*$ to be a function satisfying \eqref{eq:randomfunc}, we obtain
\begin{align}
\label{eq:sidosize1}
\nonumber \sum_{j = 0}^{l-1} |\mathcal{C}'_j|&\leq \sum_{j = 0}^{l-1} |\mathcal{C}_j| - \frac{1}{2}|\mathcal{E}| \\
\nonumber & = \left[1 + l\left(1 - \frac{1}{l}\right)^{r-1}\right] |D(G)| -  \frac{1}{2}|\mathcal{E}|\\
\nonumber & = \left[\frac{1}{2} + l\left(1 - \frac{1}{l}\right)^{r-1}\right] |D(G)| + \frac{1}{2}\left( |D(G)| - |\mathcal{E}|\right)\\
& \leq \left[\frac{1}{2} + l\left(1 - \frac{1}{l}\right)^{r-1}\right] |D(G)| + \frac{1}{2}\sum_{i = 0}^{l-1}|\mathcal{B}_i|
\end{align}
where $\mathcal{B}_i = \{\sigma \in D(G): |\sigma \cap V_i| \leq 1\}$, and 
\begin{align}
\label{eq:sidosize2}
\nonumber \E|\mathcal{B}_i| & = |D(G)| P(|\sigma \cap V_i| \leq 1)\\
& = |D(G)| \left[\left( 1 - \frac{1}{l}\right)^{r-1} + (r-1) \cdot \frac{1}{l} \left( 1 - \frac{1}{l} \right)^{r-2}\right].
\end{align}
Combining \eqref{eq:sidosize1} and \eqref{eq:sidosize2} gives
\[\E\sum_{j = 0}^{l-1} |\mathcal{C}'_j| =  \frac{|D(G)|}{2}\left[ 1 + l\left(3 + \frac{r-1}{l-1}\right)\left(1 - \frac{1}{l}\right)^{r-1}\right],\] which, by the first moment method and the pigeonhole principle, implies the assertion.
\end{proof}

\subsubsection{\texorpdfstring{Improved bounds for $d = \Theta(1)$}{Improved bounds for constant d}}
\label{sec:improvedsidorenko}
We now show that for $H$ with $d = \Theta(1)$, the cover described above can be improved. The ideas and proof are similar to the preceding sections ---  we obtain an improved bound on $\tau$ by removing additional $(r-1)$-sets from the cover constructed in the previous section --- so we aim to be brief.

For $\sigma \in \mathcal{E}$, we say $\sigma$ is {\it even} if $q(\pi(\sigma))$ is even. Clearly, for every $0 \leq i \leq l-1$, $\sigma$ is even if and only if, with $\pi(\sigma) \cap V_i = \{x, y\}$, we have
\begin{equation}
\sum_{\substack{{x_j \in \pi(\sigma) \cap V_j}\\{j \neq i}}} f(x_0, \dots, x, \dots, x_{l-1}) \equiv \sum_{\substack{{x_j \in \pi(\sigma) \cap V_j}\\{j \neq i}}} f(x_0, \dots, y, \dots, x_{l-1}) \pmod{2}.
\end{equation}

Suppose $\sigma \in \mathcal{E}$ is even, and let $e$ be an edge of the form $e = \sigma \cup \{v\}$ with $v \in V_i$. We say $e$ is \emph{compatible} with $\sigma$ if, for all $x \in \pi(\sigma) \cap V_i$,
\[ \sum_{\substack{{x_j \in \pi(\sigma) \cap V_j}\\{j \neq i}}} f(x_0, \dots, v, \dots, x_{l-1}) \equiv \sum_{\substack{{x_j \in \pi(\sigma) \cap V_j}\\{j \neq i}}} f(x_0, \dots, x, \dots, x_{l-1}) \pmod{2}. \]

Let $C_0$ be the set of $\sigma \in \mathcal{E}$ such that $\sigma$ is even, and $C_1$ be the set of $\sigma \in C_0$ such that every edge containing $\sigma$ is compatible with $\sigma$. 
As before, simply taking $C_0 \setminus C_1$ would leave some edges uncovered. Specifically, suppose $\sigma \in C_1$, $e = \sigma \cup \{v\}$ is an edge with $v \in V_i$, and $\pi(\sigma) \cap V_i = \{ x, y \}$. Then, if $e \setminus \{x\} = (\sigma \cup \{v\}) \setminus \{x\} $ and $e \setminus \{x\} = (\sigma \cup \{v\}) \setminus \{y\} $ are in $C_1$, it may be that $e$ is not covered. Let $C_2$ be the set of $\sigma \in C_1$ such that there exists an edge $e$ as in the preceding sentence containing $\sigma$.
Finally, set $C = (C_0 \setminus C_1) \cup C_2$.

Set $\mathcal{E}' = D(H) \setminus \mathcal{E}$. In words, $\mathcal{E}'$ is the set of $\sigma \in D(H)$ such that $|\sigma \cap V_i| \leq 1$ for some $0 \leq i \leq l-1$.
For each $0 \leq j \leq l-1$, let
\[ \mathcal{D}_j := \{\sigma \in \mathcal{E'} \cup C: \left( w(\sigma) + j\right) \bmod l  \in \{0, \dots, d(\sigma)\}\}. \] 
As before, we have
\begin{equation}
\label{eq:improvedsidorenko1}
\sum_{j = 0}^{l-1} |\mathcal{D}_j| = \sum_{\sigma \in \mathcal{E}' \cup C} (d(\sigma) + 1) = \sum_{\sigma \in \mathcal{E}'} (d(\sigma) + 1) + \sum_{\sigma\in C} (d(\sigma) + 1).
\end{equation}

We deal with each of the two terms on the RHS of \eqref{eq:improvedsidorenko1} separately. First note that
\begin{align*}
    \sum_{\sigma \in \mathcal{E'}} (1 + d(\sigma)) =  \sum_{\sigma \in \mathcal{E'}} 1 + \sum_{i = 0}^{l-1}  |\mathcal{A}_i|,
\end{align*}
where $\mathcal{A}_i = \{ \sigma \in D(G) : e \cap V_i = \emptyset\}$. We have
\[ \E|A_i| =  \left|D(G)\right| \left(1 - \frac{1}{l}\right)^{r-1} \sim  \binom{n}{r-1}(1 - e^{-d}) \left(1 - \frac{1}{l}\right)^{r-1}. \]
Recalling that $\zeta_2$ is the probability that, in a random partition of an $(r-1)$-set into $l$ blocks, there is a block with at most one element (i.e., $P(\sigma \in \mathcal{E}') = \zeta_2$), we obtain
\[ \sum_{\sigma \in \mathcal{E}'} 1 = \E |\mathcal{E}| = \E|D(G)| \zeta_2 \sim \binom{n}{r-1}(1-e^{-d}) \zeta_2. \]
It follows that
\begin{equation}
\label{eq:improvedsidorenko2}
    \E \sum_{\sigma \in \mathcal{E}'} (1 + d(\sigma)) \sim \binom{n}{r-1}  \left(1 - e^{-d}\right) \left[ l\left(1 - \frac{1}{l}\right)^{r-1} + \zeta_2 \right]
\end{equation}

On the other hand, for every $\sigma \in \mathcal{E}$, we have $d(\sigma) = 0$. Hence to bound the latter term on the RHS of \eqref{eq:improvedsidorenko1}, it suffices to bound the expected size of $C$. To bound the probability that $\sigma \in C$, we rely on the coupling from Section~\ref{sec:coupling} (with $\gamma = 2$). Recall that $\zeta_1$ is the probability that, in a uniform random partition of an $(r-1)$-set into $l$ blocks, there are at least two elements in each block (i.e., $P(\sigma \in \mathcal{E}) = \zeta_1$).
Now we have $P(\sigma \in C_0) = \zeta_1/2$ and 
\begin{equation}
\label{eq:largerc1}
    P(\sigma \in C_1) \sim \frac{\zeta_1}{2}\sum_{k \geq 0} \frac{e^{-d}{d^k}}{k!}\frac{1}{2^k} = \frac{\zeta_1}{2} e^{-d/2}.
\end{equation}
Given $\sigma \in C_1$ and an edge $e = \sigma \cup \{v\}$ with $v \in V_i$, suppose that $ \pi(\sigma) \cap V_i = \{x, y\}$. The probability that both $(\sigma \setminus \{x\}) \cup \{v\}$ and $(\sigma \setminus \{y\}) \cup \{v\}$, are in $C_1$ is $(e^{-d/2})^2 = e^{-d}$. It follows that
\begin{equation}
\label{eq:largerc2}
P(\sigma \in C_2) \sim \frac{\zeta_1}{2} \left[ e^{-d/2} - \exp\left(-\frac{d}{2}\left(1 + e^{-d}\right)\right)\right].
\end{equation}
Combining \eqref{eq:largerc1} and \eqref{eq:largerc2} gives
\begin{equation*}
P (\sigma \in C) \sim \frac{\zeta_1}{2} \left[1 - \exp\left(-\frac{d}{2}\left(1 + e^{-d}\right)\right) \right].
\end{equation*}
From the above discussion, we obtain
\begin{equation}
\label{eq:improvedsidorenko3}
    \E \sum_{\sigma \in C} (1 + d(\sigma))
    = \E |D(G)|P(\sigma \in C)
    \sim \binom{n}{r-1}\frac{\zeta_1}{2} \left[1 - \exp\left(-\frac{d}{2}\left(1 + e^{-d}\right)\right) \right]
\end{equation}
Finally, \eqref{eq:improvedsidorenko1}, \eqref{eq:improvedsidorenko2}, and \eqref{eq:improvedsidorenko3} imply
\begin{equation*}
    \E \sum_{j = 0}^{l-1} |\mathcal{D}_j| \sim \binom{n}{r-1}\left[  \frac{\zeta_1}{2} \left(1 - \exp\left(-\frac{d}{2}\left(1 + e^{-d}\right)\right) \right) +  \left( l\left(1 - \frac{1}{l}\right)^{r-1}\!\!\! + \zeta_2 \right) \left(1 - e^{-d}\right)  \right],
\end{equation*}
which, by with the first moment method and the pigeonhole principle, implies the assertion.

\bibliographystyle{abbrv}
\bibliography{tuza}

\appendix

\section{Proof of Theorem~\ref{thm:mainsmall}}
\label{app:smallr}

For $d = \omega(1)$, we have $|D(H)| \sim \binom{n}{r-1}$ w.h.p. With this, Theorems~\ref{thm:matching-large-p} and \ref{thm:cover-absolute} immediately imply that w.h.p.
\[\frac{\tau}{\nu} \leq \begin{cases} 
3/2 &\text{if }r = 3, \\
16/9 &\text{if }r = 4,\\
25/16 &\text{if }r = 5. \\
\end{cases} \]

For $d = \Theta(1)$, by Theorems~\ref{thm:matching-medium-p} and \ref{thm:cover-medium-p-small-r}, it suffices to show \begin{equation}
\frac{\beta_r(d)}{\alpha_r(d)/r} \leq \left\lceil\frac{r+1}{2} \right\rceil.
\end{equation}
For $r = 3$, we have
\[ \alpha_3(d) =  1 - \left(\frac{1}{2d + 1}\right)^{1/2}, \quad \mbox{ and} \quad \beta_3(d) = \frac{1}{2} \left[1 - \exp\left(-\frac{d}{2}\left(1 + e^{-d}\right)\right) \right]. \]
For $r = 4$, we have
\[ \alpha_4(d) =  1 - \left(\frac{1}{3d + 1}\right)^{1/3} , \quad \mbox{ and} \quad \beta_4(d) = \frac{4}{9} \left[1 - \exp\left(-\frac{d}{3}\left(2 + e^{-2d}\right)\right)\right]. \]
For $r = 5$, we have
\[ \alpha_5(d) =  1 - \left(\frac{1}{4d + 1}\right)^{1/4} , \quad \mbox{ and} \quad
\beta_5(d) = 
\frac{5}{16} \left[1 - \exp\left(-\frac{d}{2}\left(1 + e^{-d}\right)\right) \right]. \]

We refer to {\cite[Lemma~1.6]{kp20}} for a detailed sketch of a proof for $r = 3$. For $r = 4, 5$, the asserted bound is easy when $d \geq 5$ since $\alpha_4(d) \geq 4/27$ and $\alpha_5(d) > 5/48$ (we also have $\beta_4(d) \leq 4/9$ and $\beta_5(d) \leq 5/16$ for all $d$). It remains to deal with the intervals $d \in [1/3, 5]$ and $d \in [1/4, 5]$ for $r = 4$ and $r = 5$, respectively. For this, we simply note that a finite computation akin to {\cite[Lemma~1.6]{kp20}} suffices to verify the asserted bound.

We remark that an explicit optimization of the expressions using {\tt Mathematica} suggests that, for $r = 3, 4, 5$, the ratio $\tau/\nu \leq C_r r$ w.h.p., where $C_r < \ceil{\frac{r+1}{2}}$. We may take $C_3 = 1.976$, $C_4 = 2.883$, and $C_5 = 2.696$.

\section{Proof of Theorem~\ref{thm:mainlarge}}

\subsection{Preliminary bounds}
\label{app:largerbounds}

We first give some bounds that will make the analysis easier for larger values of $r$. Throughout this section, we restrict attention to $r \geq r_0$, for some fixed $r_0$. 

\subsubsection{\texorpdfstring{An effective bound on $\tau$}{An effective bound on tau}} \label{subsec-effective-tau}

From Theorem~\ref{thm:cover-absolute}, for any $l \leq r/2$ and any $r$-graph $G$, we have
\begin{equation}
\label{eq:t0absolute}
\tau(G) \leq \frac{|D(G)|}{2}\left[ \frac{1}{l} + \left(3 + \frac{r-1}{l-1}\right)\left(1 - \frac{1}{l}\right)^{r-1}\right].
\end{equation}

We will set
\[ l = \floor{\frac{r-1}{\ln(r-1) + \sqrt{\ln(r-1)}}},\] so
\[\frac{r-1}{\ln(r-1) + \sqrt{\ln(r-1)}} -1 \leq l \leq \frac{r-1}{\ln(r-1) + \sqrt{\ln(r-1)}}. \]
Now, using the lower bound for $l$, we have, for $r \geq r_0$,
\begin{equation}
\label{eq:approx1}
\frac{1}{l} \leq \frac{\ln(r-1) + \sqrt{\ln(r-1)} + c_0}{(r-1)}
\end{equation}
as long as
\begin{equation}
\label{eq:c0bound}
    c_0 \geq \frac{\left(\ln(r_0-1)+\sqrt{\ln(r_0-1)}\right)^2}{(r_0-1)-\left(\ln(r_0-1)+\sqrt{\ln(r_0-1)}\right)}.
\end{equation}

Similarly,
\begin{equation}
\label{eq:approx2}
\frac{r-1}{l-1} \leq \ln(r-1) + \sqrt{\ln(r-1)} + c_1
\end{equation}
holds for $r \geq r_0$ as long as
\begin{equation}
\label{eq:c1bound}
 c_1 \geq \frac{2\left(\ln(r_0-1)+\sqrt{\ln(r_0-1)}\right)^2}{(r_0-1)-2\left(\ln(r_0-1)+\sqrt{\ln(r_0-1)}\right)}.
\end{equation}

On the other hand, the upper bound for $l$ gives
\[ \left(1 - \frac{1}{l}\right)^{r-1} \leq \frac{1}{(r-1)\exp\left(\sqrt{\ln(r-1)}\right)} \]
which is valid for all $r$. Along with \eqref{eq:approx2}, this implies 
\begin{equation}
\label{eq:approx3}
\left(3 + \frac{r-1}{l-1}\right)\left(1 - \frac{1}{l}\right)^{r-1} \leq \frac{\ln(r-1) + \sqrt{\ln(r-1)} + 3 + c_1}{(r-1)\exp\left(\sqrt{\ln(r-1)}\right)} \leq \frac{c_2}{r-1},
\end{equation}
where the second inequality, for all $r \geq r_0$, is valid for
\begin{equation}
\label{eq:c2bound}
c_2 \geq \frac{\ln(r_0-1)+\sqrt{\ln(r_0-1)}+3+c_1}{e^{\sqrt{\ln(r_0-1)}}}.
\end{equation}

Finally, \eqref{eq:t0absolute}, \eqref{eq:approx1}, and \eqref{eq:approx3} together imply what is the main point of this section. For any $r \geq r_0$ and any $r$-graph $G$,
\begin{equation}
    \label{eq:tabsolute}
    \tau(G) \leq \frac{{|D(G)|}}{2} \left[\frac{\ln(r-1) + \sqrt{\ln(r-1)} + c_0 + c_2}{r-1} \right],
\end{equation} 
where $c_0$, $c_1$ and $c_2$ are as in \eqref{eq:c0bound}, \eqref{eq:c1bound}, and \eqref{eq:c2bound} respectively. In our application of \eqref{eq:tabsolute}, we will take $c_0, c_1$, and $c_2$ to be the lower end of their ranges, and, hence, treat them as functions of $r_0$.

\subsubsection{\texorpdfstring{A lower bound on $\alpha_r$ for $d \geq 2$}{A lower bound on alpha for d >= 2}} \label{subsec-alpha}

Note that, for each fixed $r$, $\alpha_r(d)$ is monotone increasing in $d$. We claim that there is $c_3 > 0$ such that, for all $r \geq r_0$ and $d \geq 2$, 
\begin{equation}
\label{eq:napprox}
\alpha_r(d) \geq 1 - \left(\frac{1}{2r-1}\right)^{1/(r-1)} \geq \frac{(1-c_3)\ln (2r-1)}{r-1}.
\end{equation}
Note that \eqref{eq:napprox} is equivalent to 
$$
\left(1- \frac{(1-c_3)\ln (2r-1)}{r-1}\right)^{r-1} \geq \frac{1}{2r-1},
$$
which, for $c_3$ large enough, is implied by 
$$
\left(e^{-\ln(2r-1)/(r-1)}\right)^{r-1} \geq \frac{1}{2r-1}.
$$
The last implication follows from the standard bound:
$$
e^{-x} \leq 1-x+\frac{x^2}{2} \leq 1 - (1-c_3)x
$$
where the second inequality is true as long as $x \leq 2c_3$. Hence \eqref{eq:napprox} holds as long as
\begin{equation}
\label{eq:c3bound}
c_3 \geq \frac{\ln(2r_0-1)}{2(r_0-1)}. \end{equation}
Finally, from \eqref{eq:napprox}, we obtain that, for every $r \geq r_0$ and $d \geq 2$,
\begin{equation}
\label{eq:nudgeq2}\alpha_r(d) \geq \frac{(1 - c_3)\ln(2r-1)}{r-1},
\end{equation}
where $c_3$ is as in \eqref{eq:c3bound}. In using \eqref{eq:nudgeq2}, we will take $c_3$ to be the lower end of its range, and, hence, treat it as a function of $r_0$.

\subsection{\texorpdfstring{Proof for $d = \omega(1)$}{Proof for for d >> 1}}
\label{app:largerlarged}

For $d = \omega(1)$, we have $D(H) \sim \binom{n}{r-1}$ w.h.p. Then, Theorems~\ref{thm:matching-large-p} and \ref{thm:cover-absolute} (with $l = 2$) give
\[ \frac{\tau}{\nu} \leq \frac{r}{2}\left[ \frac{1}{2} + (r+2)\left(\frac{1}{2}\right)^{r-1} \right] \quad \text{w.h.p}. \]
Now, that $\tau/\nu < Cr$ is implied by \[ (r+2)\left(\frac{1}{2}\right)^{r-1} < 2C - \frac{1}{2}. \]
Setting $C = 0.4$, this is evidently true for all $r \geq 6$.

For every constant $r_0$, and $r \geq r_0$, Theorem~\ref{thm:matching-large-p} and \eqref{eq:tabsolute} give
\[ \frac{\tau}{\nu} \leq \frac{r}{2(r-1)} \left(\ln(r-1) + \sqrt{\ln(r-1)} + c_0 + c_2 \right) \quad \text{w.h.p}, \]
where $c_0, c_2$ go to zero as $r_0 \rightarrow \infty$. So, for large enough $r$, we obtain the significantly stronger bound \[ \frac{\tau}{\nu} = (1 + o(1))\frac{\ln r}{2} \quad \text{w.h.p}. \]

\subsection{\texorpdfstring{Proof for for $d = \Theta(1)$}{Proof for for d ~ 1}}
\label{app:largermediumd}

For $d = \Theta(1)$, we have $|D(H)| \sim \binom{n}{r-1} (1 - e^{-d})$ w.h.p. Let $$\eta_r(d) : = \frac{1 - e^{-d}}{\alpha_r(d)}. $$ 
We first show that, for each $r \geq 6$, $\eta$ is weakly increasing in the range $d \in [1/(r-1), 2]$. This allows us to restrict our attention to $d \geq 2$, where we use the trivial bound $|D(H)| \leq \binom{n}{r-1}$ simplifying the analysis considerably (at the cost of a worse constant). 

\begin{claim}
For each fixed $r \geq 6$, $\eta_r(d)$ is weakly increasing for $1/(r-1) \leq d \leq 2$.
\end{claim}
\begin{proof}

Note that
\begin{eqnarray*}
& \frac{\partial}{\partial d} \left[\alpha_r(d)\right]  = \left(\frac{1}{(r-1)d + 1}\right)^{\frac{r}{r-1}}, \quad \mbox{ and } \quad
& \frac{\partial}{\partial d} \left[1 - e^{-d}\right] = e^{-d}.
\end{eqnarray*}
Hence, that $\eta_r$ is weakly increasing on $d \in [1/(r-1), 2]$ is implied by
\[ e^{-d} \alpha_r(d) - \left(1 - e^{-d}\right)  \left(\frac{1}{(r-1)d + 1}\right)^{\frac{r}{r-1}} \geq 0. \]
Setting $z = (r-1)d$, this is equivalent to
\[ f(z) := (z+1)^{\frac{r}{r-1}} - z - e^\frac{z}{r-1} \geq 0. \]
in the range $z \in (1, 2(r-1)]$. In the arguments below, we restrict our attention to the interval $z \in [0, \infty)$. We have 
\begin{align*}
    f'(z) & = \frac{r}{r-1}\left(z+1\right)^{\frac{1}{r-1}}-1 - \frac{1}{r-1}e^{\frac{z}{r-1}}, \quad\mbox{and }\\
    f''(z) & =  \frac{r}{\left(r-1\right)^{2}}\left(z+1\right)^{\left(\frac{1}{r-1}-1\right)} - \frac{1}{\left(r-1\right)^{2}}e^{\frac{z}{r-1}}.
\end{align*}
For $r \geq 3$, the first term is monotone decreasing and the second term is monotone increasing; hence the difference, $f''(z)$, is monotone decreasing. We also have, for $r > 1$, that $f''(0) = 1/(r-1) > 0$ and
$$ f''(2(r-1)) = \frac{1}{(r-1)^2} \left(\frac{r}{(2r-1)^{1 - \frac{1}{r-1}}} - e^2\right) < 0. $$
Indeed, the last claim follows from the fact that $r/(2r-1)^{1 - 1/(r-1)}$ is decreasing in $r$ with a limit of $1/2$. Now we have $f'(0) = 0$, that $f'(0 + \epsilon) > 0$ for small enough $\epsilon$, and that $f'$ is has a unique critical point. Finally, we have that $f(0) = 0$, that $f(0 + \epsilon) > 0$ for small enough $\epsilon$, and that $f$ has at most one critical point. Our assertion now follows from the fact that $f(2(r-1)) \geq 0$, for which it suffices to note that
\[f(2(r-1)) = (2r-1)^\frac{r}{r-1} - (r - 1) - e^2 \geq 0, \]
for $r = 5$, and that $f(2(r-1))$ is increasing in $r$.
\end{proof}

\subsubsection{\texorpdfstring{Bounds for large $r$}{Bounds for large r}}

Here we prove \eqref{eq-mainlarge2}, and  \eqref{eq-mainlarge1} for $r \geq 271$. We will restrict attention to $r \geq r_0$ and use the bounds obtained in Appendix~\ref{app:largerbounds}. Recall that, from Theorem~\ref{thm:matching-medium-p}, we have
\begin{equation*}
\nu(H) > (1 + o(1))\binom{n}{r-1} \frac{1}{r} \alpha_r(d) \quad\text{w.h.p}. 
\end{equation*}

Suppose first that $1/(r-1) \leq d \leq 2$. With $|D(H)| \sim \binom{n}{r-1} (1 - e^{-d})$ (w.h.p.), \eqref{eq:tabsolute} gives
\[ \tau(G) < (1 + o(1)) \binom{n}{r-1}\frac{1}{2} \left[\frac{\ln(r-1) + \sqrt{\ln(r-1)} + c_0 + c_2}{r-1} \right]\left(1 - e^{-d} \right) \quad\text{w.h.p}. \]
It follows that
\[ \frac{\tau}{\nu} \leq \frac{r}{2(r-1)} \left(\ln(r-1) + \sqrt{\ln(r-1)} + c_0 + c_2 \right) \eta_r(d) \quad\text{w.h.p}. \]
Since $\eta_r(d)$ is weakly increasing for $d \in [1/(r-1), 2]$, it suffices to bound this expression for $d = 2$. 
From \eqref{eq:nudgeq2}, we have
$$\alpha_r(2) \geq  \frac{(1 - c_3)\ln(2r-1)}{r-1},
$$
giving \[ \eta_r(d) = \frac{1 - e^{-d}}{\alpha_r(d)}  \leq \frac{(r-1)(1 - e^{-2})}{(1 - c_3)\ln(2r-1)}. \]
Hence, for $1/(r-1) \leq d \leq 2$, we obtain
\begin{equation*}
\label{eq:r271smalld}
\frac{\tau}{\nu} \leq r \cdot \delta_r \cdot (1 - e^{-2}) \quad\text{w.h.p}.,
\end{equation*}
where
\begin{equation*}
\delta_r := \frac{\ln(r-1) + \sqrt{\ln(r-1)} + c_0 + c_2}{2(1 - c_3)\ln(2r-1)}.\end{equation*}

For $d \geq 2$, we rely on Theorem~\ref{thm:matching-medium-p} and \eqref{eq:tabsolute} with $|D(H)| \leq \binom{n}{r-1}$ to obtain
\begin{equation*}
\label{eq:r271larged}
\frac{\tau}{\nu} \leq \delta_r \cdot r \quad\text{w.h.p}.
\end{equation*}
In either case, we have $\tau/\nu \leq \delta_r r$. We claim that $\delta_r$ is decreasing in~$r$, and, hence, it suffices to bound $\delta_{r_0}$.

\begin{claim}
For $r \geq 4$, $\delta_r$ is decreasing in $r$.
\end{claim}
\begin{proof}
By setting $z = \ln(r-1)$ and then differentiating, it suffices to show that
\[ \left(1 + \frac{1}{2\sqrt{z}}\right)\ln(2e^z + 1) - \frac{2e^z}{2e^z + 1}\left(z + \sqrt{z} + c_0 + c_2\right) < 0, \]
which is equivalent to
\begin{equation}
\label{eq:decreasing1}
\left(1 + \frac{1}{2e^z}\right) \ln(2e^z + 1) < \left(1 - \frac{1}{2\sqrt{z} + 1}\right) \left(z + \sqrt{z} + c_0 + c_2 \right).
\end{equation}
Denote by $L(z)$ (resp. $R(z)$), the LHS (resp. RHS) of \eqref{eq:decreasing1}.
We have
\[ \frac{dL}{dz} = 1 - \frac{\ln(2e^z + 1)}{2e^z}, \quad \mbox{ and }\quad \frac{dR}{dz} = 1 + \frac{z + \sqrt{z} + c_0 + c_2}{\sqrt{z}(2\sqrt{z}+ 1)^2}, \]
which implies that $L(z) < R(z)$ for $z \geq z_0$, where we may take $z_0 = 1$. The assertion follows.
\end{proof}

From the above discussion, we have that for every $r_0 \geq 6$, the ratio $\tau/\nu$ is bounded above (w.h.p.) by $\delta_{r_0} r$. Noting that $c_i(r_0) \rightarrow 0$ as $r_0 \rightarrow \infty$ for each $0 \leq i \leq 3$, we get that $\delta_{r_0} \rightarrow 0.5$ as $r_0 \rightarrow \infty$, establishing \eqref{eq-mainlarge2} of Theorem \ref{thm:mainlarge}.

To see \eqref{eq-mainlarge1} of Theorem~\ref{thm:mainlarge} for $r \geq 271$, it suffices to compute explicitly the values of the different constants. We may take $c_0 = 0.2421, c_1 = 0.50, c_2 = 1.08,$ and $c_3 = 0.012$. With this, we obtain $\delta_{271} < 0.747.$

\subsubsection{\texorpdfstring{Bounds for medium $r$}{Bounds for medium r}}

For $7 \leq r \leq 270$, we reduce the analysis to a finite computation, which can be carried out easily through {\tt Mathematica}. Computation shows that $C$ may be taken to be $0.938$. 

For $d \geq 2$, we rely on Theorem~\ref{thm:cover-absolute} with the trivial bound $D(H) \leq \binom{n}{r-1}$, and Theorem~\ref{thm:matching-medium-p} to obtain
\begin{equation}
\label{eq:mediumr1}
\frac{\tau}{\nu} \leq \frac{r}{2} \left[ \frac{1}{l} + \left(3 + \frac{r-1}{l-1}\right)\left(1 - \frac{1}{l}\right)^{r-1}\right] \frac{1}{\alpha_r(d)} \quad \text{w.h.p}.
\end{equation}
Since $\alpha_r$ is a monotone increasing function in $d$, it suffices to bound \eqref{eq:mediumr1} at $d = 2$.

For $1/(r-1) \leq d \leq 2$, we use Theorem~\ref{thm:cover-absolute} with $D(H) \sim \binom{n}{r-1}(1 - e^{-d})$, and Theorem~\ref{thm:matching-medium-p} to obtain
\begin{equation}
\label{eq:mediumr2}
    \frac{\tau}{\nu} \leq \frac{r}{2} \left[ \frac{1}{l} + \left(3 + \frac{r-1}{l-1}\right)\left(1 - \frac{1}{l}\right)^{r-1}\right] \eta_r(d) \quad \text{w.h.p}.
\end{equation}
Since $\eta_r(d)$ is monotone increasing for $1/(r-1) \leq d \leq 2$, it suffices to bound \eqref{eq:mediumr2} at $d = 2$.

In either case, we have
\begin{equation*}
    \label{eq:mediumr}
    \frac{\tau}{\nu} \leq \frac{r}{2} \left[ \frac{1}{l} + \left(3 + \frac{r-1}{l-1}\right)\left(1 - \frac{1}{l}\right)^{r-1}\right] \left( 1 - \left(\frac{1}{2r-1}\right)^{1/(r-1)} \right)^{-1} \quad \text{w.h.p}.
\end{equation*}
Now, it suffices to verify that, for each $7 \leq r \leq 270$ there exists an $l \leq r/2$ such that the expression on the RHS is bounded by  $0.938 r$. 

\section{6-uniform hypergraphs}
\label{app:r6}
We need only consider $d = \Theta(1)$. We use Theorem~\ref{thm:cover-medium-p-large-r} with $l = 2$ (and, of course, $r = 6$). It is easy to verify that
\[ \zeta_1(6, 2) = \frac{20}{32}, \quad \text{and} \quad  \zeta_2(6, 2) = \frac{12}{32}. \]
And so, w.h.p.
$$\tau(H) < (1 + o(1)) \binom{n}{5} \left[\frac{5}{32} \left(1 - \exp\left(-\frac{d}{2} \left(1 + e^{-d}\right)\right) \right) + \frac{7}{32} \left(1 - e^{-d}\right) \right].$$
On the other hand, from Theorem~\ref{thm:matching-medium-p}, we have w.h.p.
\[ \nu(H) > (1 + o(1))  \binom{n}{5} \frac{1}{6} \left[ 1 - \left(\frac{1}{5d + 1}\right)^{1/5}\right].\]
We claim that, for $d \geq 1/5$, \[ \frac{\tau}{\nu} < 0.781 \cdot 6 = 4.686. \]
This is easily verified for $d \geq 6$ (since $\tau \leq 3/8$ and $\nu$ is increasing in $d$). To show that it is true for $d \in [1/5, 6]$, one may rely on a finite computation similar to {\cite[Lemma~1.6]{kp20}}.

We note that the methods used in Section~\ref{sec:covers} can be extended to give $\tau/\nu < 0.721 * 6 = 4.326$. Unfortunately, this is not enough to settle Conjecture~\ref{conj:generaltuzar}. Hence, we only give a brief sketch below.

We may obtain an explicit 5-cover of $H$ by taking $5$-sets of type $50, 05, 14$, and $32$ (in a uniform random partition). The bound in Theorem~\ref{thm:cover-medium-p-large-r} is then obtained by removing certain 5-sets of type $32$. We may also (as in Sections~\ref{sec:cover3}--\ref{sec:cover5}) remove $5$-sets of type $41$ that only cover edges of type $51$ (edges that these cover are also covered by 5-sets of type $50$); and $5$-sets of type $05, 50$ that only cover edges contained within the same block (though one then needs to add back a small proportion of these). This results in a bound of the form
\[ \tau < (1 + o(1)) \binom{n}{5} \frac{3}{8} \left[1 - \exp\left(-\frac{d}{2}\left(1 + e^{-d}\right)\right) \right]. \]

\section{Further computations}
\label{app:computations}

The parameter $\zeta_1(r,l)$, introduced just before the statement of Theorem~\ref{thm:cover-medium-p-large-r}, can be expressed explicitly as
\begin{equation} \label{eq-prop}
\zeta_1(r,l) = \frac{1}{l^{r-1}}\sum_{a_1+\cdots+a_l = r-1,~a_i \geq 2} \binom{r-1}{a_1,\ldots,a_l},
\end{equation}
and so for $d=\Theta(1)$ (and $d \geq 1/(r-1)$), Theorems~\ref{thm:matching-medium-p} and~\ref{thm:cover-medium-p-large-r} can be combined to give an upper bound on $\tau/\nu$ that (up to a factor of $1+o_n(1)$) is an explicit expression involving $r, l$ and $d$, specifically: 
\begin{equation} \label{minmax}
\frac{\tau}{\nu} \leq (1+o(1))\min_{2 \leq l \leq r/2} \sup_{d \geq 1/(r-1)} \left(\frac{\psi_{r,l}(d)}{\frac{1}{r}\left(1-\left(\frac{1}{(r-1)d+1}\right)^{1/(r-1)}\right)}\right).
\end{equation}
The expression in \eqref{eq-prop} is not amenable to easy calculation, except for quite small $r$; however, setting $T(r,l)=l^{r-1}\zeta_1(r,l)$, we have the following much simpler explicit expression: 
\begin{equation} \label{eq-prop2}
T(r,l) = l!\sum_{j=0}^l (-1)^j\binom{r-1}{j} \stirling{r-1-j}{l-j},
\end{equation}
(see \cite[A200091]{oeis}).
Here $\stirling{a}{b}$ is the Stirling number of the second kind. Using this expression, the minimax optimization (\ref{minmax}) can be computed using {\tt Mathematica} in a reasonable amount of time, up to say $r=500$. For larger values of $r$, we can still obtain upper bounds for the ratio $\tau/\nu$ by a judicious restriction on the range of $l$. The table below shows the results obtained for $6 \leq r \leq 85$. For $86 \leq r \leq 1000$, the bound obtained on $\tau/\nu$ is always at most $Cr$ for $C<.60$. Specifically, we may take $C=0.5993$. This strongly suggests that we may take $C<0.60$ for all $r \geq 83$; and since $\delta_{1000} < 0.6964$, these computations are enough to show certainly that the ratio $\tau/\nu$ is bounded above by $Cr$ for some $C < 0.7$ for all $r \geq 12$.

\begin{center}
\begin{tabular}{|c|c||c|c||c|c||c|c|}
\hline
$r$ & $\tau/(r\nu) \leq$ & $r$ & $\tau/(r\nu) \leq$ & $r$ & $\tau/(r\nu) \leq$ & $r$ & $\tau/(r\nu) \leq$ \\
\hline
6 & 0.7805 & 26 & 0.6263 & 46 & 0.611 & 66 & 0.6058 \\
\hline 
7 & 0.7064 & 27 & 0.6253 & 47 & 0.6097 & 67 & 0.604 \\
\hline
8 & 0.6789 & 28 & 0.6266 & 48 & 0.6093 & 68 & 0.6026 \\
\hline
9 & 0.6804 & 29 & 0.6299 & 49 & 0.6098 & 69 & 0.6018 \\
\hline
10 & 0.699 & 30 & 0.6309 & 50 & 0.6109 & 70 & 0.6014 \\
\hline 
11 & 0.7062 & 31 & 0.6248 & 51 & 0.6127 & 71 & 0.6013 \\
\hline
12 & 0.6741 & 32 & 0.6209 & 52 & 0.6099 & 72 & 0.6017 \\
\hline 
13 & 0.6565 & 33 & 0.6188 & 53 & 0.608 & 73 & 0.6025 \\
\hline
14 & 0.6503 & 34 & 0.6185 & 54 & 0.6067 & 74 & 0.6027 \\
\hline
15 & 0.6526 & 35 & 0.6196 & 55 & 0.6062 & 75 & 0.6013 \\
\hline
16 & 0.6613 & 36 & 0.6221 & 56 & 0.6063 & 76 & 0.6003 \\
\hline
17 & 0.6638 & 37 & 0.622 & 57 & 0.607 & 77 & 0.5997 \\
\hline
18 & 0.6484 & 38 & 0.6178 & 58 & 0.6082 & 78 & 0.5994 \\
\hline
19 & 0.6392 & 39 & 0.6151 & 59 & 0.6075 & 79 & 0.5995 \\
\hline 
20 & 0.6353 & 40 & 0.6136 & 60 & 0.6057 & 80 & 0.5999 \\
\hline
21 & 0.6357 & 41 & 0.6134 & 61 & 0.6044 & 81 & 0.6005 \\
\hline
22 & 0.6397 & 42 & 0.6141 & 62 & 0.6037 & 82 & 0.6003 \\
\hline
23 & 0.6465 & 43 & 0.6159 & 63 & 0.6035 & 83 & 0.5992 \\
\hline
24 & 0.6369 & 44 & 0.6165 & 64 & 0.6038 & 84 & 0.5984 \\
\hline
25 & 0.63 & 45 & 0.6132 & 65 & 0.6046 & 85 & 0.5979 \\
 \hline
\end{tabular}

{\small The results of running the minimax optimization (\ref{minmax}) on {\tt Mathematica} for $6 \leq r \leq 85$.} 
\end{center}

\end{document}